%% file: topopt_els_arxiv_v2.tex
\documentclass{amsart}
\usepackage{amssymb}
\usepackage{amsthm}
\usepackage{mathtools}
\usepackage{ upgreek }

\usepackage{amssymb}
\usepackage[figuresright]{rotating}
\usepackage[ruled, linesnumbered, noend]{algorithm2e}
\usepackage{subcaption}
\usepackage[colorinlistoftodos,textwidth=4cm,shadow]{todonotes}

\newtheorem{theorem}{Theorem}

\usepackage{pgf}
\title{Robust regularization of topology optimization problems with aposteriori error estimators}
\author[skoltech]{G. V. Ovchinnikov}
\author[nyu,skoltech]{D. Zorin}
\author[skoltech]{I. V. Oseledets}

\begin{document}

\begin{abstract}

   Topological optimization finds a material density distribution minimizing a functional of the solution of a partial differential equation  (PDE), subject to a set of constraints (typically, a bound on the volume or mass of the material).

Using a finite elements discretization (FEM) of the PDE and functional we obtain an integer programming problem. 
Due to approximation error of FEM discretization, optimization problem becomes mesh-depended and possess false, physically inadequate optimums, while functional value heavily depends on fineness of discretization scheme used to compute it.
To alleviate this problem, we propose regularization of given functional by error estimate of FEM discretization. This regularization provides robustness of solutions and  improves obtained functional values as well.

While the idea is  broadly applicable, in this paper we apply our method to the  heat conduction optimization. This type of problems are of  practical importance in design of heat conduction channels, heat sinks and other types of heat guides.

\end{abstract}
\maketitle

\section{Introduction}
Topology optimization methods aim to find a material distribution in a volume that minimizes a functional depending on the solution of a PDE describing a physical process, e.g., elastic deformation, heat distribution or electrical current flow. 

As in general these problems are noncovex, and even non-smooth,  most methods can be viewed as heuristics and do not guarantee reaching a global or local optimum. 
The solutions produced by commonly used techniques may exhibit substantial mesh dependence with significant difference in functional values obtained for different mesh resolutions, often due to non-physical solutions (``checkerboards''). These solutions are eliminated by different types of regularization techniques, penalizing non-smoothness of the solutions.

While some mesh dependence is unavoidable for discretized problems, we observe that undesirable mesh-dependent solutions are the ones for which the FEM discretization does not yield an accurate estimate of the functional. Informally, the less ``smooth'', in the sense of presence of rapid oscillations, the discrete solutions is, the larger is the range of functional values that can be attained by functions agreeing with the discrete solution at sample points. 

This view suggests the following idea for regularization: instead of choosing a measure of smoothness \emph{a priori},
we directly include an estimate of the range of possible functional values in the regularized functional, by adding a
term corresponding to the FEM error. While the approximation error can not be measured directly,
computable upper bounds are typically available.  Thus, instead of minimizing the value of the discretized functional, we minimize the upper bound of the true functional. 

More specifically, we use \emph{aposteriori estimates} to modify the original functionals.
This allows us to perform topological optimization using standard sensitivity-based greedy algorithms using a fixed (non-adaptive) discretization. which is important for efficient implementation of the method (including much simpler parallelization).

To summarize, our approach consists of the following elements:
\begin{itemize}
    \item We propose using \emph{a posteriori} error estimators as reliable regularizers for topological optimization.
\item We obtain a weak FEM problem formulation to compute the sensitivity of our regularized functionals which can be solved efficiently using a standard FEM solver; 
\item We describe a greedy optimization algorithm based on the computed sensitivity to compute the optimal topology. 
\item For two model problems, we obtained new configurations that have significantlylower functional values that the configurations reported in the literature.
\end{itemize}
Although we focus on a particular model problem in heat transfer, our approach can be easily generalized to other topology optimization settings.

We start with the description of the method (Sections~\ref{sec:continuous}-\ref{sec:optimization}) and
then compare it to closely related methods and discuss its relationship to most common topology optimization techniques (Section~\ref{sec:experiments},\ref{sec:discussion}).

\section{Continuous problem}
\label{sec:continuous}
As a model problem in this paper we consider the steady heat conduction problem:
\begin{equation} 
    \label{to:heat}
    -\nabla \cdot (w \nabla u) = f, \quad u_{\partial \Omega} = 0,
\end{equation}
where $\Omega$ is a bounded region with a sufficiently smooth boundary $\partial \Omega$, $w$ is in the set of 
piecewise-continuous functions, which we denote $PC$, $f \in L^2(\Omega)$,
and we search for $u \in V$, where $V$ is the Sobolev space $W^1_2(\Omega)$ of functions with zero trace (due to the boundary conditions). 

We solve a two-material optimization problem, i.e.,  $w$ is allowed to take only two values: $1$ (corresponding to the main material) and $\varepsilon \ll 1$  (corresponding to the filler material), with $\varepsilon$ being  the ratio of conductivity of filler material to that of main material. Our model topology optimization problem is formulated as follows: find a material distribution $w$ such that
\begin{equation}
\begin{aligned}
& \underset{w \in PC}{\text{min}}
        & & F(w) = \int_{\Omega} u(x) f(x) \, dx = (u,f), \\
& \text{s. t.}
& & \int w(x) \, dx \geq c, \quad u \text{ is solution of (\ref{to:heat})}.
\end{aligned}
\label{to:optproblem}
\end{equation}

A physical interpretation of this optimization problem is the design of  optimal heat-conducting device, producing least amount of heat when amount of high-conductivity material is limited.  The same mathematical formulations appears not only in heat conduction problems, but also in electrostatic problems \cite{steven2000evolutionary} or modelling of amoeboid organism growing towards food sources \cite{safonov}.

Note that in general, the problem does not have an optimal solution in $V$; regularized versions of the problem however, may have sufficiently smooth solutions. The analysis of existence and smoothness of solutions of the optimization problem is beyond the scope of this paper. 

\section{Discretization}
\label{sec:discrete}
A common way to solve  problem (\ref{to:heat}) is to introduce in $\Omega$ a mesh $\Omega_h$, consisting of elements $\Omega_h^i$, and replace $w$ by its projection  $w_h$
to the finite-dimensional space of functions $\theta_i$, which are piecewise constant on the elements:
$$w_h(x) = \sum_{i=1}^n \eta^h_i \theta_i(x),\,\eta^h_i \in \{\varepsilon,1\},\,\theta_i(x)  \in \{0,1\}, $$
where $n$ is number of mesh cells and  $\theta_i(x)$ is 1 on $i$-th cell, and zero otherwise.
Problem \eqref{to:heat} can be written in the following variational form:
\begin{equation}
  \label{to:bl}
  B(u, v) = l(v), \quad \forall v \in V, 
\end{equation} 
where $B(u,v)$ is a bilinear form on $V \times V$ and $l(v)$ is a linear functional on 
dual space:
$$
B(u,v) = \int_{\Omega} w \nabla u \nabla v \, d\Omega, \quad
l(v) =\int_{\Omega} f v \, d\Omega.
$$
We obtain finite element discretization by replacing $V$ in \eqref{to:bl} by a suitable finite-dimensional space $V_h$.
This leads us to the following system of linear equations for $u_h \in V_h$:
\begin{equation}
    \label{to:fem1}
  B(u_h, v_h) = l(v_h), \quad \forall v_h \in V_h,  
\end{equation}
From now on, we will assume $f$ is a function for which $f = f_h$.

The optimization problem \eqref{to:optproblem} is then approximated by the following integer programming problem:
\begin{equation}
    \begin{aligned}
        & \underset{\eta^h_i \in \{\varepsilon,1\}^n}{\text{min} }
            F_h(w_h) = (u_h, f_h), \quad  \text{s.t.}
            \sum_{i=1}^{n} {\eta^h}_i \geq c, \quad u_h \text{ is solution of \eqref{to:fem1}}.
\end{aligned}
\label{to:intprog}
\end{equation}
The solutions to the discretized problem may exhibit spurious behavior, becoming increasingly oscillatory, with non-physical ``checkerboard'' patterns emerging, with oscillations on the scale of mesh elements. The emergence of these patterns in the case of elasticity is attributed to the fact that on the scale of the mesh, the discretization exhibit artificial stiffness resulting in a significant underestimation of the functional value.  We propose a new way of tackling this type of problems by introducing an additional term into the functional that penalizes $\eta$ that give rises to large errors in the solution. 

\section{\emph{A posteriori} error estimators as regularizers}
\label{sec:regularizer}
Transition from the initial optimization problem \eqref{to:optproblem} to its discrete counterpart
\eqref{to:intprog} introduces a discretization error, which can be estimated from above as
$$
(u, f) - (u_h, f_h)  \leq \| u_h - u \| \| f \|.
$$

The major contribution of this paper is in consideration of the following regularized functional:
\begin{equation}
\label{F_reg}
        F_h (\alpha, w_h) = F_h(w_h) + \alpha \Uptheta_h (w_h),
\end{equation}
where the regularization term $\Uptheta_h(\eta)$ is an upper bound for the error of solution given by $\eta$. We assume that for some $\alpha > \alpha_0$ it provides an upper bound for the true functional:
$$
  F(w_h) \leq F_h(\alpha, w_h).
$$
  Our idea is to minimize the upper bound $F_h(\alpha, w_h)$, instead of the discretization of the functional itself. While the value of $F_h(w_h)$ for the resulting solution may be larger, one can guarantee that $F(w_h)$ is not arbitrarily large, unlike the case of standard greedy methods, which may lead to solutions for which  $F(w_h)$ can be significantly larger than $F_h(w_h)$.

  There are many ways to introduce the regularizer term based on the error estimation; we present a simple error estimator bounding the squared error: 
  $$\Vert u - u_h \Vert^2 \leq \Uptheta_h(w_h)$$
  to ensure that the regularized functional is differentiable everywhere.
  We obtain a simple regularizer from from the following variational form \cite{gratsch2005}. Let $e = u - u_h$, then 
$$
  B(e, q) = l(q) - B(u_h,q),\,\, \forall q \in V. 
$$
For  Galerkin discretizations, the error is orthogonal to $V_h$, and hence we need to choose  
subspace $V_h' \neq V_h$ to solve for the error:
\begin{equation}
    \label{weak_err}
        B(\tilde{e}_h, q_h) = l(q_h) - B(u_h,q_h),\,\, \forall q_h \in V_h'. 
\end{equation}

Choosing $\Uptheta_h(w_h) = (\tilde{e}_h, \tilde{e}_h)$ leads us to the regularized version of (\ref{to:intprog}):
\begin{equation}
\begin{aligned}
& \underset{\eta_h \in \{0,1\}^n}{\text{min}}
        & & F_h(\alpha, w_h) = (u_h, f_h) + \alpha (\tilde{e}_h, \tilde{e}_h) \\
& \text{s. t.}
        & & \sum_i \eta^h_i \geq c,\quad u_h \text{ is a solution of (\ref{to:fem1})},\quad  \tilde{e}_h \text{ is solution of (\ref{weak_err}).}
\end{aligned}
\label{to:regintprog}
\end{equation}

\section{Optimization problem}
\label{sec:optimization}
\subsection{Greedy optimization}
To solve \eqref{to:intprog} we use the approach that has been successfully used in the topology optimization: the greedy 
``hard-kill'' methods introduced in \cite{xie1993simple}, \cite{xie1997} and \cite{querin1998},
and used more recently in \cite{gao2008}). Those methods remove/add a whole patch of material on each step, in contrast to the ``soft-kill'' methods which evolve a smooth material density function which is discretized in the final step by thresholding and/or filtering.
One can easily see that both continuous and discrete functionals $F(w)$ and $F_h(w_h)$ are monotonous with respect to changes in $w$, 
so ``removing material'' (i.e., replacing $1$ by $\varepsilon$) always increases it, because 
boundary conditions are chosen such that $u$ to be non-negative.
The greedy approach to solving this problem removes the material from an element
that results in least change to the functional. To estimate the change in the functional, we use its \emph{sensitivity}, i.e.,  the derivative of the functional with respect to the coefficient $\eta^h_i$ of the material distribution corresponding to the $i$-th element:
\begin{equation}\label{to:sensitivity}
        S_i = \frac{\partial F_h}{\partial \eta_i}, \quad i=1, \ldots, n, 
\end{equation}
An efficient way to compute sensitivities will be described in Section \ref{to:sens_sec} 
Algorithm \ref{al:greedy} summarizes the complete method.
\begin{algorithm}
 \SetAlgoLined
        \KwData{Total material bound $c$, number of patches to remove at each step $p_s$, mesh $\mathcal{M}$, low material density $\varepsilon$, regularization parameter $\alpha$}
        \KwResult{Vector of the material distribution coefficients $\eta^h$ and the value of $F_h$}
    Fill the  whole mesh $\mathcal{M}$  by setting $\eta_i = 1, i = 1,\dots,n$\; 
        \While {$\sum_i \eta_i > c$}{
    Calculate sensitivities of regularized functional $\partial{F_{h, \alpha}}/\partial \eta_i$\;
    Choose up to $p_s$ elements with smallest sensitivities\;
    Set those elements density to $\varepsilon$\;
        }
    Calculate $F_h$\;
    \Return{$F_h$, $\eta$}
\caption{Greedy method}
\label{al:greedy}
\end{algorithm}

For a given problem we can vary the number of elements $n$ and mesh connectivity, which affects the material distribution and functional value. We will discuss the  effects of those parameters in numerical experiments section \ref{sssec:params}.

\subsection{Computing sensitivity}\label{to:sens_sec}
The derivatives of the regularized functional consists of two terms, corresponding to the original functional and the regularization term.  The expression for the former is well-known (we provide the derivation for the sake of completeness); we derive the expression for the regularizer.

We represent the form $B$ in finite-dimensional spaces $V$ and $V'$ by symmetric stiffness matrices $K$ and $M$: $B(u_h,v_h) = (Ku_h, v_h)$, and $B(u_h',v_h') = (Mu_h', v_h')$, which leads us to the following equations for $u_h$ and $u_h'$:
\begin{equation}
\label{FEM_K}
Ku_h = f_h,
\end{equation}
\begin{equation}
\label{FEM_M}
Mu_h' = f_h'.
\end{equation}

We use the same notation $u_h$  for functions and their coefficient vectors, as the  difference is clear from the context
(the coefficient vectors are used only in matrix equations).

Matrices $P$ and $P'$ project a function in $V_h$ to a function in $V_h'$ and vice versa.
To simplify derivation, for the rest of the paper, we assume that $f_h' = f_h = f$ and $V_h \in V_h'$, for example $V_h'$ being generated by the same basis function family, but on finer mesh, hence $P$ is the inclusion operator.

The discretized unregularized functional can be written as $F_h = u_h^T f_h$; Using
$K_h u_h = f_h$, where $f_h$ is independent of $\eta$, we obtain 

$$  \frac{\partial{F_h}}{{\partial \eta^h_i}} = \left( \frac{\partial K^{-1}}{\partial \eta^h_i}f_h, f_h \right), 
$$
From $KK^{-1} = I$ it follows
$$
\frac{\partial{K}}{\partial \eta^h_i} K^{-1}+ \frac{\partial{K^{-1}}}{\partial \eta^h_i}K = 0.$$
We conclude that
\begin{equation}
  \frac{\partial{F_h}}{{\partial \eta_i}} = -\left(\frac{\partial{K}}{\partial \eta_i} u_h, u_h\right).
\label{eq:unreg-sensitivity}
\end{equation}

\subsubsection{Sensitivity of the regularization term}
Now, we have two natural choices for error estimate $\tilde{e}_h$ from (\ref{to:regintprog}).
First, $e_h = u_h - P'u_h$, and second, $e_h' = Pu_h - u_h'$.
\begin{theorem}
Sensitivity of the functional $F_{h,\alpha}$ regularized by $e_h$ is
\begin{equation}
\label{eh_sens}
\frac{\partial F_{h,\alpha}}{\partial \eta_i} = -\left(\frac{\partial{K}}{\partial \eta_i} u_h, u_h\right) -2 \alpha \left(P'M^{-1}\frac{\partial M}{\partial \eta_i}u_h', e_h\right) + 2\alpha  \left(\frac{\partial K}{\partial \eta_i}u_h, K^{-1}e_h\right). 
\end{equation}
Sensitivity of the functional $F_{h,\alpha}$ regularized by $e_h'$ is
\begin{equation}
\label{eh_dash_sens}
\frac{\partial F_{h,\alpha}}{\partial \eta_i} = -\left(\frac{\partial{K}}{\partial \eta_i} u_h, u_h\right) -2 \alpha \left(\frac{\partial M}{\partial \eta_i}u_h', M^{-1}e_h'\right) + 2\alpha  \left(PK^{-1}\frac{\partial K}{\partial \eta_i}u_h, e_h'\right). 
\end{equation}
\end{theorem}
\begin{proof}
The first part of the formulas is given by \eqref{eq:unreg-sensitivity}.  
For the second part,  notice that from \eqref{FEM_K} and \eqref{FEM_M} follows: 
\[
        e_h = K^{-1}f_h - P'M^{-1}f_h'.
\]
Then sensitivity of $e_h$ can be written as
\[
        \frac{\partial e_h}{\partial \eta_i} = -K^{-1}\frac{\partial K}{\partial \eta_i}K^{-1}f_h + P'M^{-1}\frac{\partial M}{\partial \eta_i}M^{-1}f_h',
\]
and hence
\[
        \frac{\partial (e_h, e_h)}{\partial \eta_i} = -2(\frac{\partial K}{\partial \eta_i}u_h, K^{-1}e_h) + 2(P'M^{-1}\frac{\partial M}{\partial \eta_i}u_h', e_h). 
\]
This, combined with \eqref{eq:unreg-sensitivity} proves \eqref{eh_sens}.
Now, consider $e_h'$:
\[
        e_h' = PK^{-1}f_h - M^{-1}f_h'.
\]
Then sensitivity of $e_h'$ can be written as
\[
        \frac{\partial e_h'}{\partial \eta_i} = -PK^{-1}\frac{\partial K}{\partial \eta_i}K^{-1}f_h + M^{-1}\frac{\partial M}{\partial \eta_i}M^{-1}f_h',
\]
and hence
\[
        \frac{\partial (e_h', e_h')}{\partial \eta_i} = -2(PK^{-1}\frac{\partial K}{\partial \eta_i}u_h, e_h') + 2(\frac{\partial M}{\partial \eta_i}u_h', M^{-1}e_h'). 
\]
This, combined with \eqref{eq:unreg-sensitivity} proves \eqref{eh_dash_sens}.

\end{proof}

\subsection{Computational details}
To compute $(\partial{K}/\partial \eta_i u_h, u_h)$,  $(\partial M /\partial \eta_i\, u_h', M^{-1}e_h')$ and
$(\partial K /\partial \eta_i\,u_h, K^{-1}e_h)$ 
observe that
\begin{equation}
\label{int_sens1}
\left( \frac{\partial K}{\partial \eta_i} u_h, v_h \right) = \int_{\Omega} \theta_i \nabla u_h \nabla v_h \, d\Omega,\, i=1,\dots, n,
\end{equation}
\begin{equation}
\label{int_sens2}
\left( \frac{\partial M}{\partial \eta_i} u_h', v_h' \right) = \int_{\Omega} \theta_i \nabla u_h' \nabla v_h'\, d\Omega,\,\, i=1,\dots, n.
\end{equation}
Using these expressions we compute sensitivities by solving for $u_h$ and $u_h'$
using a FEM solver, computing their gradients, and  then compute integrals (\ref{int_sens1}) and (\ref{int_sens2}).  If an optimal-complexity linear solver is used (e.g., multigrid) the overall cost of the computation is linear in the total number of elements
in $V_h$ and $V'_h$.

\section{Numerical experiments}
\label{sec:experiments}
\subsection{Test cases}

Although  topology optimization by greedy methods has been studied extensively, there are relatively few papers
considering topology optimization for heat conduction. We use \cite{gao2008} and \cite{gersborg2006} for comparison
purposes.

We consider two model 2D problems, both of each have been studied in literature, which
allows us to provide comparisons with other methods: problem A was considered
in \cite{gersborg2006} and Problem B in \cite{zuo2005} and \cite{gao2008}.

Both problems are formulated on a simple square domain of size 1, and a uniform heat distribution $f$ with unit density
is assumed.

\textbf{Problem A.} For this problem,  Dirichlet boundary conditions are specified on two boundary segments
and  Neumann conditions   on the other two (Figure~\ref{fig:bc}, left).
The target volume fraction $c$ occupied by the material for this problem is set to $0.4$, to be consistent with
\cite{gersborg2006}.

\textbf{Problem B.} For this problem, Dirichlet conditions are specified on the whole boundary  (Figure~\ref{fig:bc}),
and the target volume fraction $c$ occupied by the material for this problem is set to $0.5$,
following  \cite{zuo2005} and \cite{gao2008}.

\begin{figure}[!htb]
    \centering
\begin{subfigure}{0.49\linewidth} 
\centering
\resizebox{!}{0.6\textwidth}{
\includegraphics{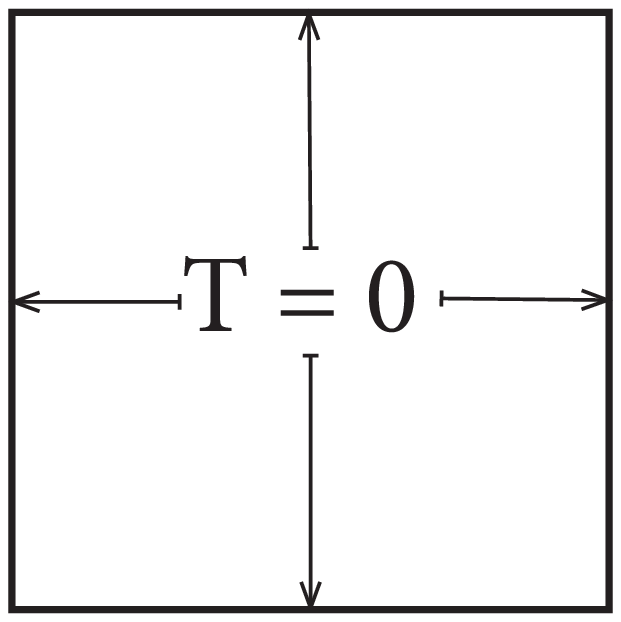}
}

\subcaption{Problem A}
\end{subfigure}
\centering
\begin{subfigure}{0.49\linewidth} 
\centering
\resizebox{!}{0.6\textwidth}{
\includegraphics{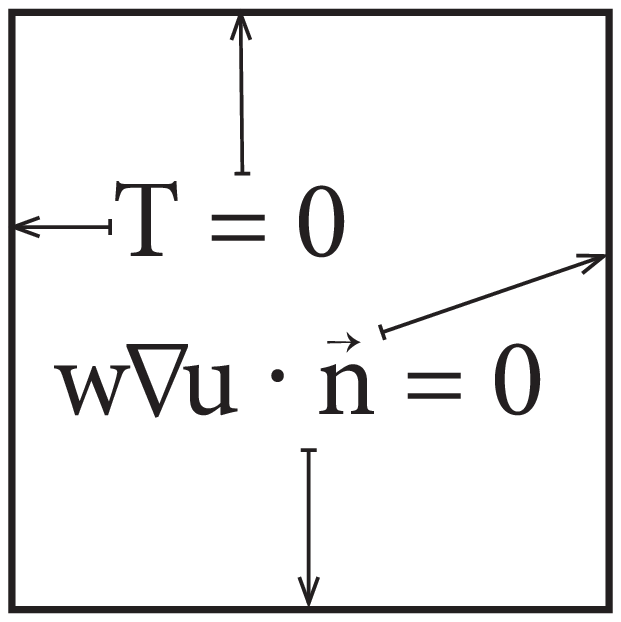}
}
\subcaption{Problem B}
\end{subfigure}
        \caption{Boundary conditions.}
\label{fig:bc}
\end{figure}

Our main comparison measure is the \emph{relative increase in the functional value}
$$
\widehat{F}_h = \frac{F_h}{F_0},
$$
where $F_0$ is the value of the functional corresponding to the whole domain filled with material, i.e., $\eta^h_i = 1$ for all elements.  
Since the functional always increases with the removal of the material, $\widehat{F}_h \geq 1$ the smaller this number is, the better is the result. 

\subsection{Implementation}
We implemented Algorithm \ref{al:greedy} using FireDrake finite element solver (\cite{FD2015}, \cite{luporini2016}, \cite{luporini2015}, \cite{mcrae2014}). 
We use the regularized functional (\ref{F_reg}) and expressions for sensitivities (\ref{int_sens1}), (\ref{int_sens2}).
While for numerical experiments we consider only squares, the approach, obviously, works
for arbitrary sufficiently regular geometry (the problem must be solvable with FEM, obviously) both in 2D and 3D.

For our experiments we choose $V_h$ to be the standard space of piecewise-linear continuous functions 
on triangular mesh;  $w_h$ is represented by piecewise-constant functions on triangles. The
space $V_h'$ needed for the error bound is also the space of piecewise-linear continuous functions on the mesh  $\Omega_h'$ which is a uniformly refined version of original mesh $\Omega_h$. In the experiments we use three different meshes generated by GMSH \cite{gmsh} (Figure \ref{fig:meshes}).
Also, since we don't need regularizer to be very precise  
for simplicity of implementation we used
\[
        \frac{\partial F_{h,\alpha}}{\partial \eta_i} = -\left(\frac{\partial{K}}{\partial \eta_i} u_h, u_h\right) -2 \alpha \left(\frac{\partial M}{\partial \eta_i}u_h', M^{-1}e_h'\right) + 2\alpha  \left(\frac{\partial K}{\partial \eta_i}u_h, K^{-1}e_h\right),
\]
as compromise between \eqref{eh_sens} and \eqref{eh_dash_sens} which does not require 
projections.
\begin{figure}[!htb]
    \centering
\centering
\begin{subfigure}{0.32\linewidth} 
\centering
\resizebox{!}{0.9\textwidth}{
\includegraphics{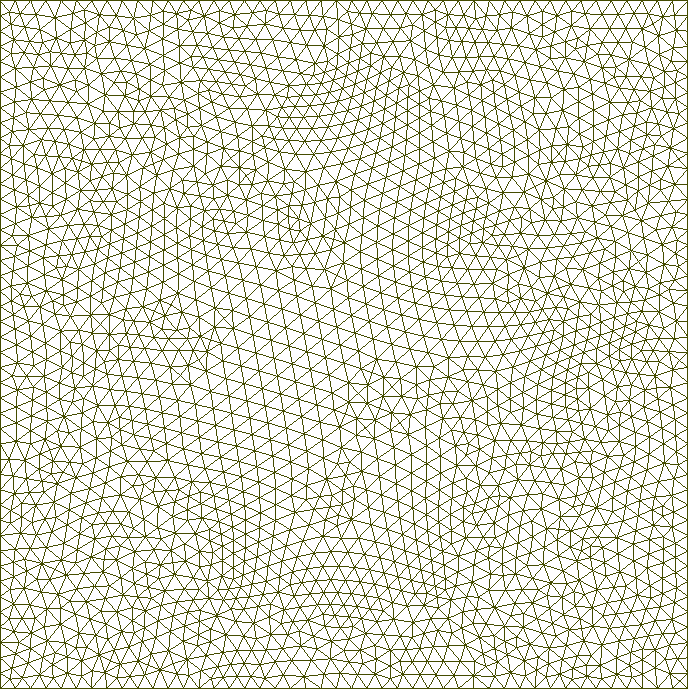}
}
        \subcaption{4736 elements}
\end{subfigure}
\begin{subfigure}{0.32\linewidth}
\centering
\resizebox{!}{0.9\textwidth}{
\includegraphics{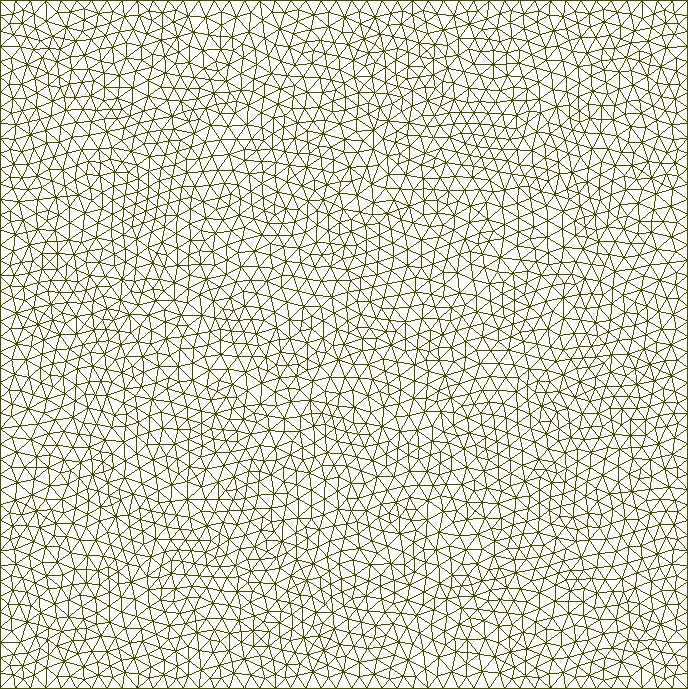}
}
        \subcaption{5578 elements}
\end{subfigure}
\begin{subfigure}{0.32\linewidth} 
\centering
\resizebox{!}{0.9\textwidth}{
\includegraphics{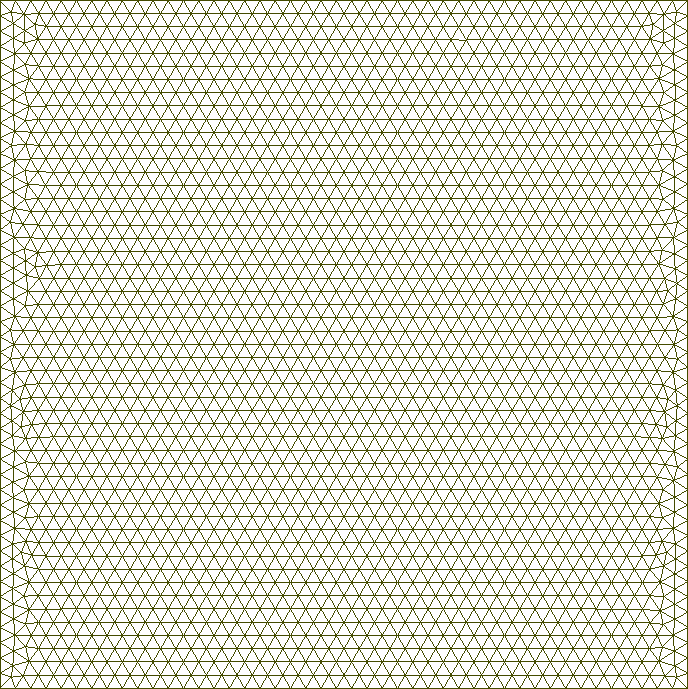}
}
        \subcaption{4900 elements}
\end{subfigure}
\caption{Left to right: Meshes 1,2,3 used in experiments.} 
\label{fig:meshes}
\end{figure}

\subsubsection{Dependence on regularization parameter and mesh-dependence} \label{sssec:params}
We first explore the dependence of the final functional value on the regularization parameter $\alpha$,
in the range from $10^{-5}$ until $10^{12}$. Simultaneously, we explore the mesh dependence of the method
by running  Algorithm \ref{al:greedy} on all combinations of meshes and regularization parameters.
\begin{figure}[!htb]
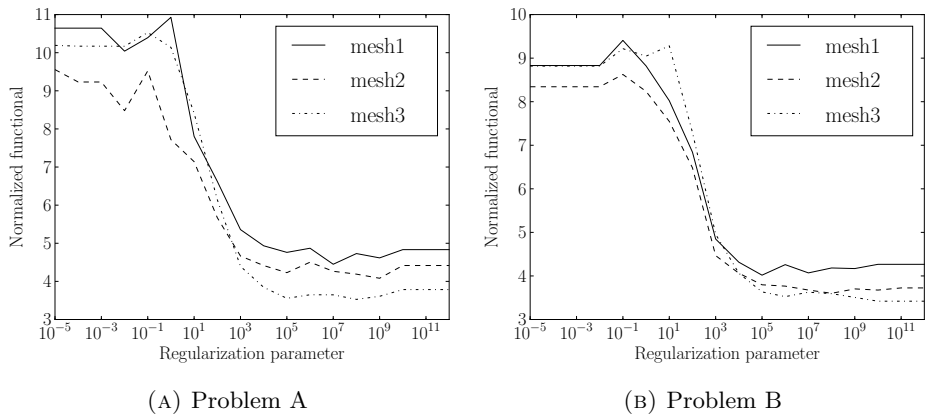

        \begin{subfigure}{0.49\linewidth}
        \resizebox{!}{0.8\textwidth}{
\input{Gersborg_func_over_alpha.pgf}
}
        \caption{Problem A}
        \end{subfigure}
        \begin{subfigure}{0.49\linewidth}
        \resizebox{!}{0.8\textwidth}{
\input{Gao_func_over_alpha.pgf}
}
        \caption{Problem B}
        \end{subfigure}
\caption{Dependence of functional on the regularization parameter $\alpha$.}
\label{fig:alpha_dep}
\end{figure}

As evident from the Figure \ref{fig:alpha_dep}, there is a noticeable dependence on the mesh choice at this resolution;
At the same time, it is also clear that the method is relatively insensitive to the choice of $\alpha$, as long as it is
sufficiently large: and the effects of the choice of $\alpha$ are not mesh-dependent, and are similar for both problems. 

Observe that the optimal values of $\alpha$ are large: effectively, minimization of the error needs to be prioritized
over minimizing the discrete functional value for best results. Intuitively, one can interpret this as searching for
optimal designs in the constrained space of functions minimizing the error on a finer mesh.
As evident from Figure~\ref{fig:refinement} functional converges with respect to mesh size 
very fast, with only one level of refinement sufficient.

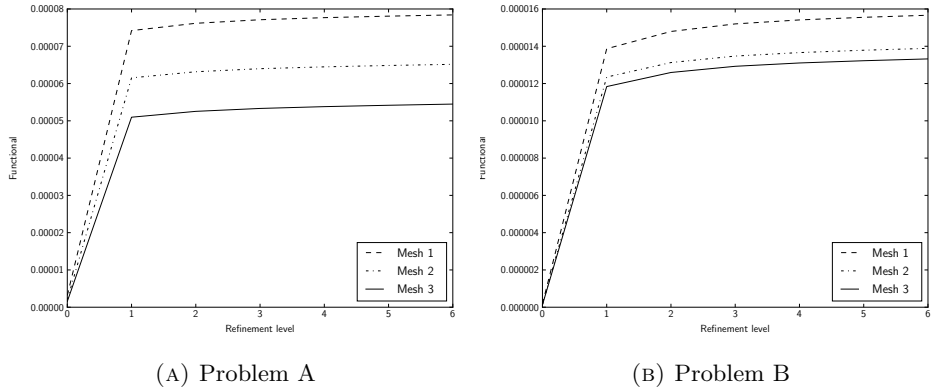
\begin{figure}[!htb]
        \begin{subfigure}{0.49\linewidth}
        \resizebox{!}{0.8\textwidth}{
\input{Gersborg_multilevel.pgf}
}
        \caption{Problem A}
        \end{subfigure}
        \begin{subfigure}{0.49\linewidth}
        \resizebox{!}{0.8\textwidth}{
\input{Gao_multilevel.pgf}
}
        \caption{Problem B}
        \end{subfigure}
\caption{Dependence of functional on number of times mesh was refined for configurations obtained by greedy algorithm with $\alpha=10^5$.}
\label{fig:refinement}
\end{figure}

\subsection{Problem A}
Based on the results shown in Figure~\ref{fig:alpha_dep}, we use $\alpha = 10^8$ in the remaining experiments. 
The material distribution found by our algorithm can be seen in Figure~\ref{fig:ger}, (c).
The obtained material distribution features thin heat-sink channels, similar to the ones from solutions of other authors.

A fundamental difference between our method and the method of \cite{gersborg2006} is that we solve the problem with
discrete material values directly, while \cite{gersborg2006} solves a relaxation of the problem with a continuous material
distribution. This approach has significant advantages as, due to convexity \cite{lau2001convex}, such problems are easier
to solve. In addition, by searching a larger space of material distributions, lower values of the functional can be obtained. However, as in practice creating continuous material variation is typically difficult or impossible, a final thresholding step
needs to be applied, which typically leads to a significant increase in the functional. 

\begin{figure}[!htb]
\begin{subfigure}{0.32\linewidth}
\centering
\resizebox{!}{\textwidth}{
\includegraphics{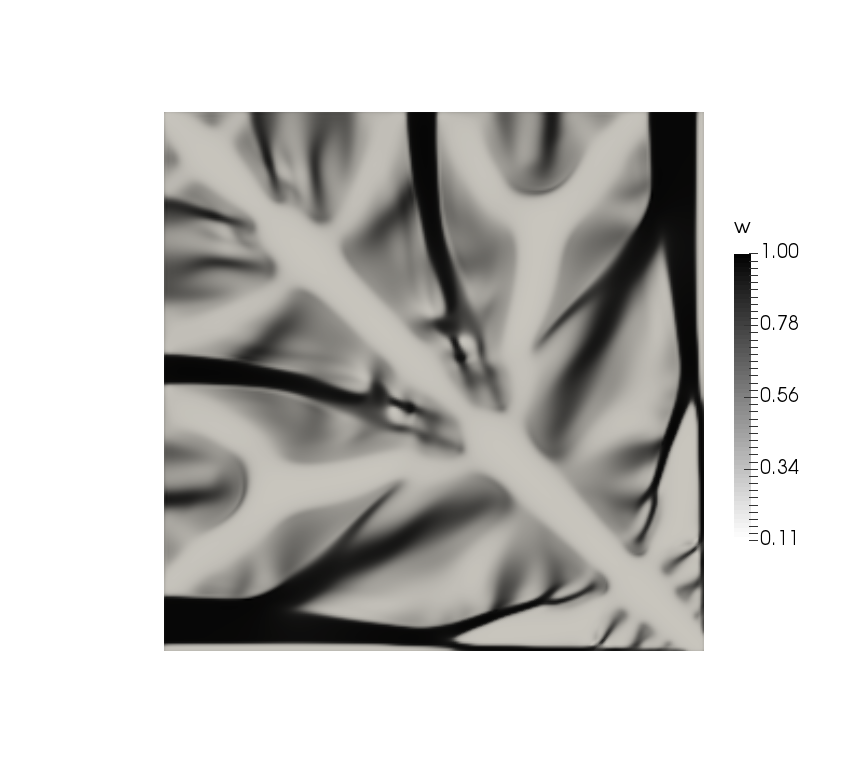}
}
\subcaption{$\widehat{F}_h  = 2.455$. SIMP solution.}
\end{subfigure}
\begin{subfigure}{0.32\linewidth} 
\centering
\resizebox{!}{\textwidth}{
\includegraphics{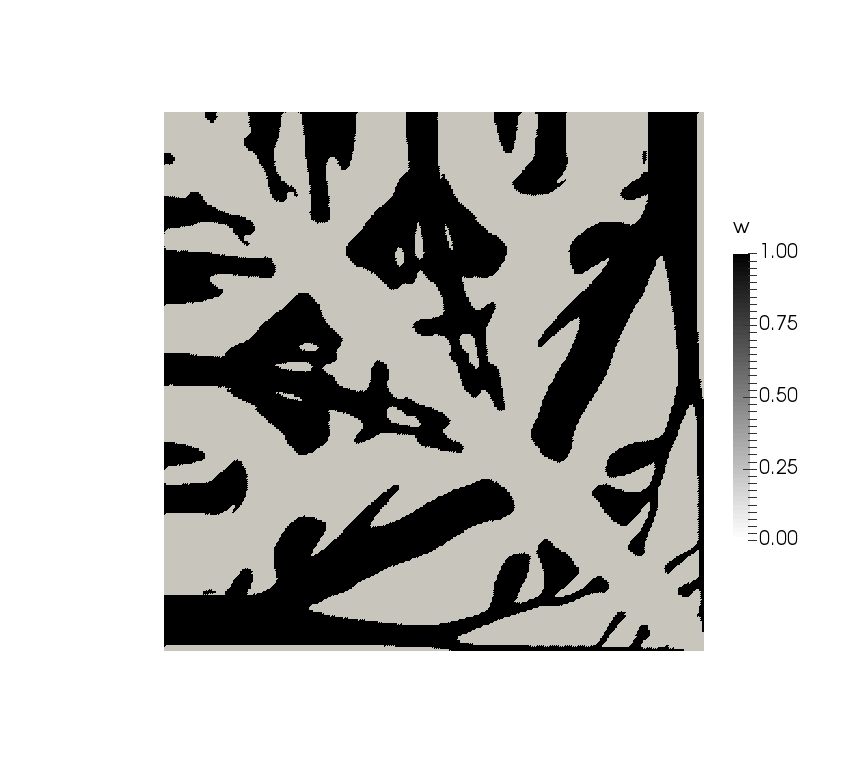}
}
\subcaption{$\widehat{F}_h  = 7.42$. Binarized solution}
\end{subfigure} 
\begin{subfigure}{0.32\linewidth} 
\centering
\vskip 9mm
\resizebox{!}{0.7\textwidth}{
\includegraphics{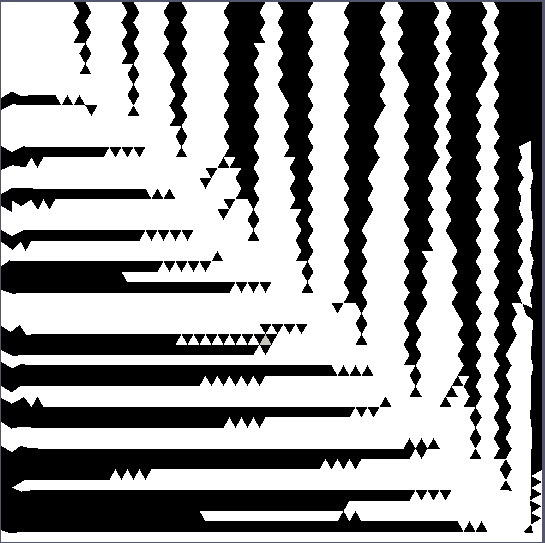}
}
\vskip 5mm
\subcaption{$\widehat{F}_h  = 3.528$,\newline $4898$ elements, mesh type 3, $\alpha = 10^8$}
\end{subfigure}
\caption{Thresholding SIMP solution results in a significant increase in the functional. } 
\label{fig:ger}
\end{figure}

Our solution appears to be less fractal  than solution from \cite{gersborg2006}.
Note that \cite{gersborg2006} considers continuous optimization problem, as material density allowed to be in $[\varepsilon, 1]$.

Aggregated results for Problem A are shown in Table \ref{gersborg_table}:
\begin{center} 
  \captionof{table}{Results and comparisons} \label{gersborg_table}
  \begin{tabular}{ | l | c | c | c | c |}
  \hline
          Paper & Elements & Solution type &  $\widehat{F}_h $ \\ 
  \hline
          This &   14694 & Discrete &  3.53 \\ \hline
          \cite{gersborg2006}  & 16384 & Continuous &  2.718\\ \hline
          \cite{gersborg2006} with  \cite{sigmund2001}  & 16384 & Continuous & 3.259 \\ \hline
\end{tabular}
\end{center}

\subsection{Problem B}
Solution of this optimization problem features the same 
thin heat-sink channels as with problem A. Our solution looks different 
from solution of \cite{gao2008}, which presents smaller number of larger channels 
with fractal-like border. On the other hand, our solution propose more thin channels
with relatively smooth borders. The jigsaw borders of heat sinks are due to 
mesh effects (compare smoother borders of the channels in top and bottom versus rougher borders of the channels in left and right).
The comparison results are shown in Table \ref{gao_table}.
\begin{figure}[!htb]
    \centering
\centering
\begin{subfigure}{0.32\linewidth} 
\centering
\resizebox{!}{0.9\textwidth}{
\includegraphics{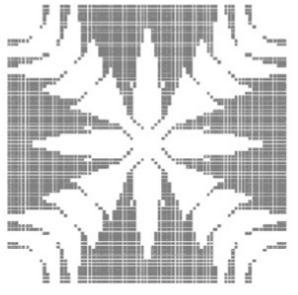}
}
\subcaption{$\widehat{F}_h  = 8.256$, quadrilateral mesh with 6400 elements.}
\end{subfigure}
\begin{subfigure}{0.32\linewidth}
\centering
\resizebox{!}{0.9\textwidth}{
\includegraphics{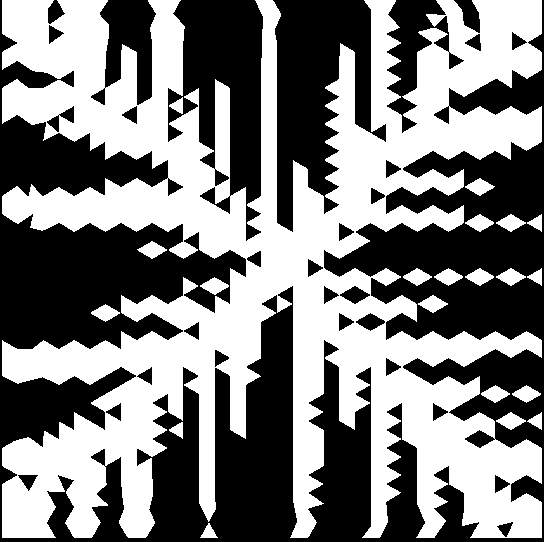}
}
\subcaption{$\widehat{F}_h  = 4.401$, mesh 2 with 2466 elements, \newline $\alpha = 10^5$}
\end{subfigure}
\begin{subfigure}{0.32\linewidth} 
\centering
\resizebox{!}{0.9\textwidth}{
\includegraphics{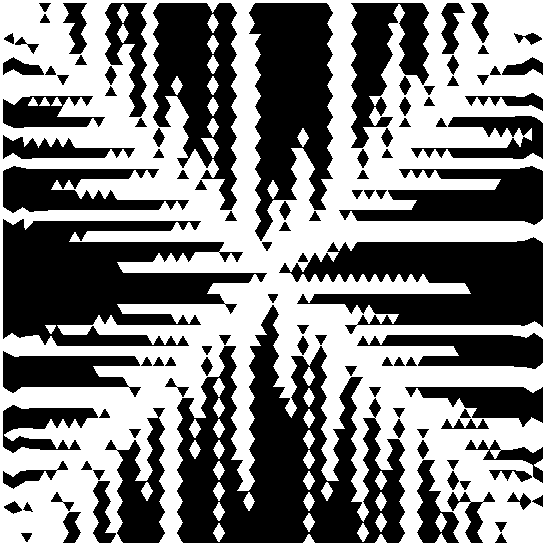}
}
        \subcaption{$\widehat{F}_h  = 3.42$, mesh 3 with 4898 elements, \newline $\alpha = 10^{12}$.}
\end{subfigure}

\caption{Left to right: material distribution from \cite{gao2008}, \cite{zuo2005} and ours for problem B.} 
\label{fig:gao}
\end{figure}

\begin{center}
  \captionof{table}{Results and comparisons} \label{gao_table}
  \begin{tabular}{ | l | c | c | c | c |}
  \hline
          Paper &  Elements & Solution type & $\widehat{F}_h$  \\ 
  \hline
          This & 14694 & Discrete &  $3.42$ \\ \hline
          \cite{gao2008} & 6400 & Discrete  & $8.256$ \\ \hline
\end{tabular}
\end{center}

\begin{figure}[!htb]
    \centering
\begin{subfigure}{0.49\linewidth} 
\centering
\resizebox{!}{0.8\textwidth}{
\includegraphics{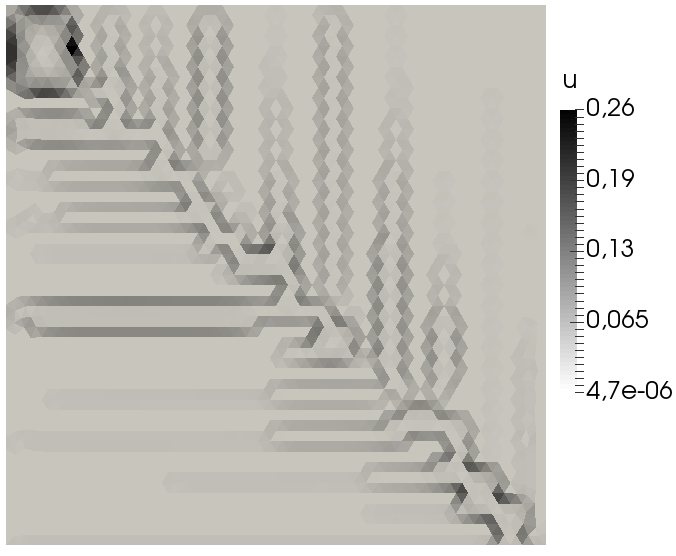}
}
\subcaption{Solution}
\end{subfigure}
\centering
\begin{subfigure}{0.49\linewidth} 
\centering
\resizebox{!}{0.8\textwidth}{
\includegraphics{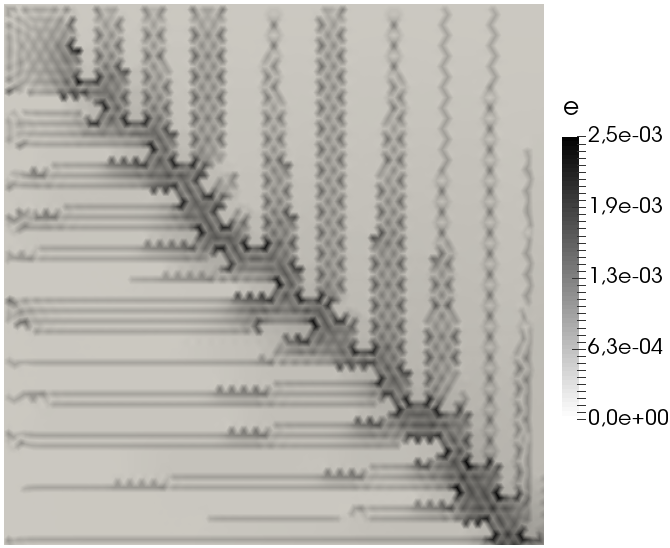}
}
\subcaption{Error}
\end{subfigure}
\caption{Solution and error estimates for problem A.}
\label{fig:ours}
\end{figure}

\section{Discussion and related work}
\label{sec:discussion}
\subsection{Previous approaches to regularization}
Regularization techniques were used in topological optimization primarily to solve two problems: to avoid formation 
of spurious solutions (checkerboard patterns) and to remove mesh dependency. Regularization was applied in the context
of all commonly used approaches to topology optimization: density-based,  hard-kill and level set/phase-field methods.

Most regularization approaches can be roughly divided into two groups: filtering methods and constraint techniques \cite{deaton2014survey}.

There are several types of filtering methods: the two most widely used types are \emph{density filters}  which directly modify the solution (\cite{bourdin2001filters}, \cite{bruns2001topology}),  and  \emph{sensitivity filters} (\cite{sigmund1997}, \cite{sigmund2001}). Recently filtering methods based on the Helmholtz equation were applied to both sensitivity and material density 
(\cite{lazarov2011filters}, \cite{kawamoto2011heaviside}). Another recent development is the
design of density filters which are based on projection schemes. Those methods operate on density field $\rho$, its filtered 
version $\tilde{\rho}$ and a projected field $\hat{\tilde{\rho}}$. Most filtering methods result in
continuous density distributions, which need to be thresholded in the context of hard-kill methods, 
the projection methods reported to have good convergence and almost discrete designs 
(\cite{guest2004achieving}, \cite{sigmund2007morphology}, \cite{xu2010volume}, \cite{guest2011eliminating},
\cite{duhring2010design}).

Constraint methods modify optimization problem by imposing a constraint either as hard constraint, or via a penalty term
in the objective functional.  A common type constrained quantity is an integral of the form
$$
\int_{\Omega}\|\nabla w \|_qd\Omega,
$$
i.e., a $L_q$-norm of the gradient. E.g., $q=1$ leads to the total variation constraint or penalty \cite{ambrosio1993optimal}, \cite{haber1996perimeter},  $q=2$ results in the usual $L_2$ gradient norm (Dirichlet energy) and  $q=\infty$ yields a pointwise constraint on the gradient magnitude \cite{niordson1983optimal}, \cite{petersson1998slope}. In \cite{borrvall2001topology},
proposed a similar penalty based on density discrepancy. Regularization for level-set methods can be found, e.g., in  \cite{allaire2004structural}, where the primary goal of regularization is to smooth the velocity field and maintain smoothness of the level set.   Detailed reviews of relevant literature can be found in two excellent surveys: \cite{deaton2014survey} and \cite{sigmund2013topology}. 

In contrast to most previous methods, our approach aims to use regularization to control the error of the FEM solution directly,
and in this way, increase robustness of the method, by potentially increasing the value of the discretized functional for
the solution, but ensuring that it does not deviate too much from the underlying smooth functional.

Our experiments demonstrate however, that in the context of hard-kill methods, this regularization actually results in a \emph{decrease} of the (unregularized) functional value.
This happens because approximation error introduced by finite elements methods
creates false minimums for objective functional. 
First, greedy methods will likely to get stuck in those minimums. 
Second those false minimums depend on mesh and result in growth of 
functional on different or refined mesh.
Regularization removes those false minimums and mesh dependence, making 
functional value being closer to its true physical value.

\begin{figure}[!htb]
    \centering
\begin{subfigure}{0.49\linewidth} 
\centering
\resizebox{!}{0.8\textwidth}{
\includegraphics{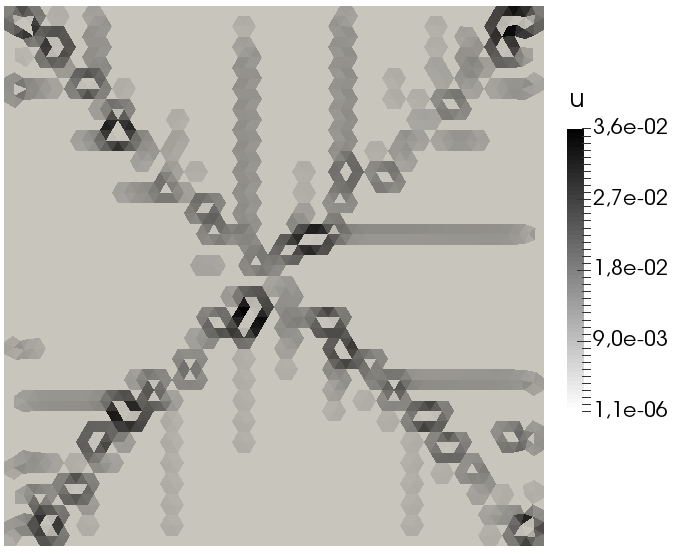}
}
\subcaption{Solution}
\end{subfigure}
\centering
\begin{subfigure}{0.49\linewidth} 
\centering
\resizebox{!}{0.8\textwidth}{
\includegraphics{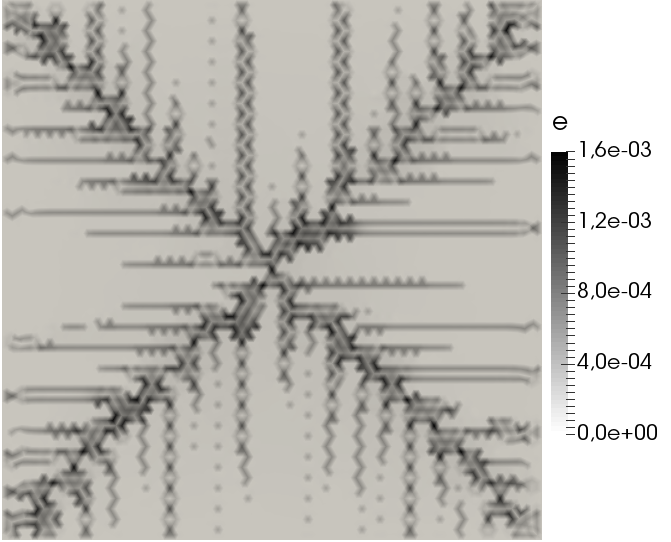}
}
\subcaption{Error}
\end{subfigure}
\caption{Solution and error estimation for problem B.}
\label{fig:ours}
\end{figure}

\section{Conclusions and future work}

Greedy methods based on sensitivity analysis are popular due to their ease of  implementation on top of existing FEM packages. 
On the downside, the integer programming problems resulted from discretization are complex, and the convergence behavior
is often unpredictable. 

In our experiments we have demonstrated that our regularization applied to ``hard-kill'' method allows it achieve results close to  those of a density-based method from \cite{gersborg2006}, and a significant improvement compared to a previously proposed
``hard-kill'' method \cite{gao2008}, with the functional value two times less.

The estimator used to regularize our functional requires additional finer mesh and solution of additional problem on that mesh.
Although it is very easy to implement, this incurs an additional cost. 
In the future, we will explore regularization with a more efficient (e.g., adaptive) error estimator, along with applications of approach proposed in this paper to topology optimization in the context of elasticity as well as to 3D problems.

\bibliographystyle{elsarticle-num}
\bibliography{references}







\end{document}

%% file: Gersborg_multilevel.pgf
\begingroup%
\makeatletter%
\begin{pgfpicture}%
\pgfpathrectangle{\pgfpointorigin}{\pgfqpoint{8.000000in}{6.000000in}}%
\pgfusepath{use as bounding box, clip}%
\begin{pgfscope}%
\pgfsetbuttcap%
\pgfsetmiterjoin%
\definecolor{currentfill}{rgb}{1.000000,1.000000,1.000000}%
\pgfsetfillcolor{currentfill}%
\pgfsetlinewidth{0.000000pt}%
\definecolor{currentstroke}{rgb}{1.000000,1.000000,1.000000}%
\pgfsetstrokecolor{currentstroke}%
\pgfsetdash{}{0pt}%
\pgfpathmoveto{\pgfqpoint{0.000000in}{0.000000in}}%
\pgfpathlineto{\pgfqpoint{8.000000in}{0.000000in}}%
\pgfpathlineto{\pgfqpoint{8.000000in}{6.000000in}}%
\pgfpathlineto{\pgfqpoint{0.000000in}{6.000000in}}%
\pgfpathclose%
\pgfusepath{fill}%
\end{pgfscope}%
\begin{pgfscope}%
\pgfsetbuttcap%
\pgfsetmiterjoin%
\definecolor{currentfill}{rgb}{1.000000,1.000000,1.000000}%
\pgfsetfillcolor{currentfill}%
\pgfsetlinewidth{0.000000pt}%
\definecolor{currentstroke}{rgb}{0.000000,0.000000,0.000000}%
\pgfsetstrokecolor{currentstroke}%
\pgfsetstrokeopacity{0.000000}%
\pgfsetdash{}{0pt}%
\pgfpathmoveto{\pgfqpoint{1.000000in}{0.600000in}}%
\pgfpathlineto{\pgfqpoint{7.200000in}{0.600000in}}%
\pgfpathlineto{\pgfqpoint{7.200000in}{5.400000in}}%
\pgfpathlineto{\pgfqpoint{1.000000in}{5.400000in}}%
\pgfpathclose%
\pgfusepath{fill}%
\end{pgfscope}%
\begin{pgfscope}%
\pgfpathrectangle{\pgfqpoint{1.000000in}{0.600000in}}{\pgfqpoint{6.200000in}{4.800000in}} %
\pgfusepath{clip}%
\pgfsetbuttcap%
\pgfsetroundjoin%
\pgfsetlinewidth{1.003750pt}%
\definecolor{currentstroke}{rgb}{0.000000,0.000000,0.000000}%
\pgfsetstrokecolor{currentstroke}%
\pgfsetdash{{6.000000pt}{6.000000pt}}{0.000000pt}%
\pgfpathmoveto{\pgfqpoint{1.000000in}{0.788903in}}%
\pgfpathlineto{\pgfqpoint{2.033333in}{5.052236in}}%
\pgfpathlineto{\pgfqpoint{3.066667in}{5.169775in}}%
\pgfpathlineto{\pgfqpoint{4.100000in}{5.226879in}}%
\pgfpathlineto{\pgfqpoint{5.133333in}{5.260496in}}%
\pgfpathlineto{\pgfqpoint{6.166667in}{5.284704in}}%
\pgfpathlineto{\pgfqpoint{7.200000in}{5.304890in}}%
\pgfusepath{stroke}%
\end{pgfscope}%
\begin{pgfscope}%
\pgfpathrectangle{\pgfqpoint{1.000000in}{0.600000in}}{\pgfqpoint{6.200000in}{4.800000in}} %
\pgfusepath{clip}%
\pgfsetbuttcap%
\pgfsetroundjoin%
\pgfsetlinewidth{1.003750pt}%
\definecolor{currentstroke}{rgb}{0.000000,0.000000,0.000000}%
\pgfsetstrokecolor{currentstroke}%
\pgfsetdash{{3.000000pt}{5.000000pt}{1.000000pt}{5.000000pt}}{0.000000pt}%
\pgfpathmoveto{\pgfqpoint{1.000000in}{0.720230in}}%
\pgfpathlineto{\pgfqpoint{2.033333in}{4.291589in}}%
\pgfpathlineto{\pgfqpoint{3.066667in}{4.389945in}}%
\pgfpathlineto{\pgfqpoint{4.100000in}{4.438917in}}%
\pgfpathlineto{\pgfqpoint{5.133333in}{4.468680in}}%
\pgfpathlineto{\pgfqpoint{6.166667in}{4.490735in}}%
\pgfpathlineto{\pgfqpoint{7.200000in}{4.509448in}}%
\pgfusepath{stroke}%
\end{pgfscope}%
\begin{pgfscope}%
\pgfpathrectangle{\pgfqpoint{1.000000in}{0.600000in}}{\pgfqpoint{6.200000in}{4.800000in}} %
\pgfusepath{clip}%
\pgfsetrectcap%
\pgfsetroundjoin%
\pgfsetlinewidth{1.003750pt}%
\definecolor{currentstroke}{rgb}{0.000000,0.000000,0.000000}%
\pgfsetstrokecolor{currentstroke}%
\pgfsetdash{}{0pt}%
\pgfpathmoveto{\pgfqpoint{1.000000in}{0.693053in}}%
\pgfpathlineto{\pgfqpoint{2.033333in}{3.658259in}}%
\pgfpathlineto{\pgfqpoint{3.066667in}{3.752404in}}%
\pgfpathlineto{\pgfqpoint{4.100000in}{3.799265in}}%
\pgfpathlineto{\pgfqpoint{5.133333in}{3.828296in}}%
\pgfpathlineto{\pgfqpoint{6.166667in}{3.850340in}}%
\pgfpathlineto{\pgfqpoint{7.200000in}{3.869409in}}%
\pgfusepath{stroke}%
\end{pgfscope}%
\begin{pgfscope}%
\pgfsetrectcap%
\pgfsetmiterjoin%
\pgfsetlinewidth{1.003750pt}%
\definecolor{currentstroke}{rgb}{0.000000,0.000000,0.000000}%
\pgfsetstrokecolor{currentstroke}%
\pgfsetdash{}{0pt}%
\pgfpathmoveto{\pgfqpoint{1.000000in}{5.400000in}}%
\pgfpathlineto{\pgfqpoint{7.200000in}{5.400000in}}%
\pgfusepath{stroke}%
\end{pgfscope}%
\begin{pgfscope}%
\pgfsetrectcap%
\pgfsetmiterjoin%
\pgfsetlinewidth{1.003750pt}%
\definecolor{currentstroke}{rgb}{0.000000,0.000000,0.000000}%
\pgfsetstrokecolor{currentstroke}%
\pgfsetdash{}{0pt}%
\pgfpathmoveto{\pgfqpoint{7.200000in}{0.600000in}}%
\pgfpathlineto{\pgfqpoint{7.200000in}{5.400000in}}%
\pgfusepath{stroke}%
\end{pgfscope}%
\begin{pgfscope}%
\pgfsetrectcap%
\pgfsetmiterjoin%
\pgfsetlinewidth{1.003750pt}%
\definecolor{currentstroke}{rgb}{0.000000,0.000000,0.000000}%
\pgfsetstrokecolor{currentstroke}%
\pgfsetdash{}{0pt}%
\pgfpathmoveto{\pgfqpoint{1.000000in}{0.600000in}}%
\pgfpathlineto{\pgfqpoint{7.200000in}{0.600000in}}%
\pgfusepath{stroke}%
\end{pgfscope}%
\begin{pgfscope}%
\pgfsetrectcap%
\pgfsetmiterjoin%
\pgfsetlinewidth{1.003750pt}%
\definecolor{currentstroke}{rgb}{0.000000,0.000000,0.000000}%
\pgfsetstrokecolor{currentstroke}%
\pgfsetdash{}{0pt}%
\pgfpathmoveto{\pgfqpoint{1.000000in}{0.600000in}}%
\pgfpathlineto{\pgfqpoint{1.000000in}{5.400000in}}%
\pgfusepath{stroke}%
\end{pgfscope}%
\begin{pgfscope}%
\pgfsetbuttcap%
\pgfsetroundjoin%
\definecolor{currentfill}{rgb}{0.000000,0.000000,0.000000}%
\pgfsetfillcolor{currentfill}%
\pgfsetlinewidth{0.501875pt}%
\definecolor{currentstroke}{rgb}{0.000000,0.000000,0.000000}%
\pgfsetstrokecolor{currentstroke}%
\pgfsetdash{}{0pt}%
\pgfsys@defobject{currentmarker}{\pgfqpoint{0.000000in}{0.000000in}}{\pgfqpoint{0.000000in}{0.055556in}}{%
\pgfpathmoveto{\pgfqpoint{0.000000in}{0.000000in}}%
\pgfpathlineto{\pgfqpoint{0.000000in}{0.055556in}}%
\pgfusepath{stroke,fill}%
}%
\begin{pgfscope}%
\pgfsys@transformshift{1.000000in}{0.600000in}%
\pgfsys@useobject{currentmarker}{}%
\end{pgfscope}%
\end{pgfscope}%
\begin{pgfscope}%
\pgfsetbuttcap%
\pgfsetroundjoin%
\definecolor{currentfill}{rgb}{0.000000,0.000000,0.000000}%
\pgfsetfillcolor{currentfill}%
\pgfsetlinewidth{0.501875pt}%
\definecolor{currentstroke}{rgb}{0.000000,0.000000,0.000000}%
\pgfsetstrokecolor{currentstroke}%
\pgfsetdash{}{0pt}%
\pgfsys@defobject{currentmarker}{\pgfqpoint{0.000000in}{-0.055556in}}{\pgfqpoint{0.000000in}{0.000000in}}{%
\pgfpathmoveto{\pgfqpoint{0.000000in}{0.000000in}}%
\pgfpathlineto{\pgfqpoint{0.000000in}{-0.055556in}}%
\pgfusepath{stroke,fill}%
}%
\begin{pgfscope}%
\pgfsys@transformshift{1.000000in}{5.400000in}%
\pgfsys@useobject{currentmarker}{}%
\end{pgfscope}%
\end{pgfscope}%
\begin{pgfscope}%
\pgftext[x=1.000000in,y=0.544444in,,top]{\sffamily\fontsize{12.000000}{14.400000}\selectfont 0}%
\end{pgfscope}%
\begin{pgfscope}%
\pgfsetbuttcap%
\pgfsetroundjoin%
\definecolor{currentfill}{rgb}{0.000000,0.000000,0.000000}%
\pgfsetfillcolor{currentfill}%
\pgfsetlinewidth{0.501875pt}%
\definecolor{currentstroke}{rgb}{0.000000,0.000000,0.000000}%
\pgfsetstrokecolor{currentstroke}%
\pgfsetdash{}{0pt}%
\pgfsys@defobject{currentmarker}{\pgfqpoint{0.000000in}{0.000000in}}{\pgfqpoint{0.000000in}{0.055556in}}{%
\pgfpathmoveto{\pgfqpoint{0.000000in}{0.000000in}}%
\pgfpathlineto{\pgfqpoint{0.000000in}{0.055556in}}%
\pgfusepath{stroke,fill}%
}%
\begin{pgfscope}%
\pgfsys@transformshift{2.033333in}{0.600000in}%
\pgfsys@useobject{currentmarker}{}%
\end{pgfscope}%
\end{pgfscope}%
\begin{pgfscope}%
\pgfsetbuttcap%
\pgfsetroundjoin%
\definecolor{currentfill}{rgb}{0.000000,0.000000,0.000000}%
\pgfsetfillcolor{currentfill}%
\pgfsetlinewidth{0.501875pt}%
\definecolor{currentstroke}{rgb}{0.000000,0.000000,0.000000}%
\pgfsetstrokecolor{currentstroke}%
\pgfsetdash{}{0pt}%
\pgfsys@defobject{currentmarker}{\pgfqpoint{0.000000in}{-0.055556in}}{\pgfqpoint{0.000000in}{0.000000in}}{%
\pgfpathmoveto{\pgfqpoint{0.000000in}{0.000000in}}%
\pgfpathlineto{\pgfqpoint{0.000000in}{-0.055556in}}%
\pgfusepath{stroke,fill}%
}%
\begin{pgfscope}%
\pgfsys@transformshift{2.033333in}{5.400000in}%
\pgfsys@useobject{currentmarker}{}%
\end{pgfscope}%
\end{pgfscope}%
\begin{pgfscope}%
\pgftext[x=2.033333in,y=0.544444in,,top]{\sffamily\fontsize{12.000000}{14.400000}\selectfont 1}%
\end{pgfscope}%
\begin{pgfscope}%
\pgfsetbuttcap%
\pgfsetroundjoin%
\definecolor{currentfill}{rgb}{0.000000,0.000000,0.000000}%
\pgfsetfillcolor{currentfill}%
\pgfsetlinewidth{0.501875pt}%
\definecolor{currentstroke}{rgb}{0.000000,0.000000,0.000000}%
\pgfsetstrokecolor{currentstroke}%
\pgfsetdash{}{0pt}%
\pgfsys@defobject{currentmarker}{\pgfqpoint{0.000000in}{0.000000in}}{\pgfqpoint{0.000000in}{0.055556in}}{%
\pgfpathmoveto{\pgfqpoint{0.000000in}{0.000000in}}%
\pgfpathlineto{\pgfqpoint{0.000000in}{0.055556in}}%
\pgfusepath{stroke,fill}%
}%
\begin{pgfscope}%
\pgfsys@transformshift{3.066667in}{0.600000in}%
\pgfsys@useobject{currentmarker}{}%
\end{pgfscope}%
\end{pgfscope}%
\begin{pgfscope}%
\pgfsetbuttcap%
\pgfsetroundjoin%
\definecolor{currentfill}{rgb}{0.000000,0.000000,0.000000}%
\pgfsetfillcolor{currentfill}%
\pgfsetlinewidth{0.501875pt}%
\definecolor{currentstroke}{rgb}{0.000000,0.000000,0.000000}%
\pgfsetstrokecolor{currentstroke}%
\pgfsetdash{}{0pt}%
\pgfsys@defobject{currentmarker}{\pgfqpoint{0.000000in}{-0.055556in}}{\pgfqpoint{0.000000in}{0.000000in}}{%
\pgfpathmoveto{\pgfqpoint{0.000000in}{0.000000in}}%
\pgfpathlineto{\pgfqpoint{0.000000in}{-0.055556in}}%
\pgfusepath{stroke,fill}%
}%
\begin{pgfscope}%
\pgfsys@transformshift{3.066667in}{5.400000in}%
\pgfsys@useobject{currentmarker}{}%
\end{pgfscope}%
\end{pgfscope}%
\begin{pgfscope}%
\pgftext[x=3.066667in,y=0.544444in,,top]{\sffamily\fontsize{12.000000}{14.400000}\selectfont 2}%
\end{pgfscope}%
\begin{pgfscope}%
\pgfsetbuttcap%
\pgfsetroundjoin%
\definecolor{currentfill}{rgb}{0.000000,0.000000,0.000000}%
\pgfsetfillcolor{currentfill}%
\pgfsetlinewidth{0.501875pt}%
\definecolor{currentstroke}{rgb}{0.000000,0.000000,0.000000}%
\pgfsetstrokecolor{currentstroke}%
\pgfsetdash{}{0pt}%
\pgfsys@defobject{currentmarker}{\pgfqpoint{0.000000in}{0.000000in}}{\pgfqpoint{0.000000in}{0.055556in}}{%
\pgfpathmoveto{\pgfqpoint{0.000000in}{0.000000in}}%
\pgfpathlineto{\pgfqpoint{0.000000in}{0.055556in}}%
\pgfusepath{stroke,fill}%
}%
\begin{pgfscope}%
\pgfsys@transformshift{4.100000in}{0.600000in}%
\pgfsys@useobject{currentmarker}{}%
\end{pgfscope}%
\end{pgfscope}%
\begin{pgfscope}%
\pgfsetbuttcap%
\pgfsetroundjoin%
\definecolor{currentfill}{rgb}{0.000000,0.000000,0.000000}%
\pgfsetfillcolor{currentfill}%
\pgfsetlinewidth{0.501875pt}%
\definecolor{currentstroke}{rgb}{0.000000,0.000000,0.000000}%
\pgfsetstrokecolor{currentstroke}%
\pgfsetdash{}{0pt}%
\pgfsys@defobject{currentmarker}{\pgfqpoint{0.000000in}{-0.055556in}}{\pgfqpoint{0.000000in}{0.000000in}}{%
\pgfpathmoveto{\pgfqpoint{0.000000in}{0.000000in}}%
\pgfpathlineto{\pgfqpoint{0.000000in}{-0.055556in}}%
\pgfusepath{stroke,fill}%
}%
\begin{pgfscope}%
\pgfsys@transformshift{4.100000in}{5.400000in}%
\pgfsys@useobject{currentmarker}{}%
\end{pgfscope}%
\end{pgfscope}%
\begin{pgfscope}%
\pgftext[x=4.100000in,y=0.544444in,,top]{\sffamily\fontsize{12.000000}{14.400000}\selectfont 3}%
\end{pgfscope}%
\begin{pgfscope}%
\pgfsetbuttcap%
\pgfsetroundjoin%
\definecolor{currentfill}{rgb}{0.000000,0.000000,0.000000}%
\pgfsetfillcolor{currentfill}%
\pgfsetlinewidth{0.501875pt}%
\definecolor{currentstroke}{rgb}{0.000000,0.000000,0.000000}%
\pgfsetstrokecolor{currentstroke}%
\pgfsetdash{}{0pt}%
\pgfsys@defobject{currentmarker}{\pgfqpoint{0.000000in}{0.000000in}}{\pgfqpoint{0.000000in}{0.055556in}}{%
\pgfpathmoveto{\pgfqpoint{0.000000in}{0.000000in}}%
\pgfpathlineto{\pgfqpoint{0.000000in}{0.055556in}}%
\pgfusepath{stroke,fill}%
}%
\begin{pgfscope}%
\pgfsys@transformshift{5.133333in}{0.600000in}%
\pgfsys@useobject{currentmarker}{}%
\end{pgfscope}%
\end{pgfscope}%
\begin{pgfscope}%
\pgfsetbuttcap%
\pgfsetroundjoin%
\definecolor{currentfill}{rgb}{0.000000,0.000000,0.000000}%
\pgfsetfillcolor{currentfill}%
\pgfsetlinewidth{0.501875pt}%
\definecolor{currentstroke}{rgb}{0.000000,0.000000,0.000000}%
\pgfsetstrokecolor{currentstroke}%
\pgfsetdash{}{0pt}%
\pgfsys@defobject{currentmarker}{\pgfqpoint{0.000000in}{-0.055556in}}{\pgfqpoint{0.000000in}{0.000000in}}{%
\pgfpathmoveto{\pgfqpoint{0.000000in}{0.000000in}}%
\pgfpathlineto{\pgfqpoint{0.000000in}{-0.055556in}}%
\pgfusepath{stroke,fill}%
}%
\begin{pgfscope}%
\pgfsys@transformshift{5.133333in}{5.400000in}%
\pgfsys@useobject{currentmarker}{}%
\end{pgfscope}%
\end{pgfscope}%
\begin{pgfscope}%
\pgftext[x=5.133333in,y=0.544444in,,top]{\sffamily\fontsize{12.000000}{14.400000}\selectfont 4}%
\end{pgfscope}%
\begin{pgfscope}%
\pgfsetbuttcap%
\pgfsetroundjoin%
\definecolor{currentfill}{rgb}{0.000000,0.000000,0.000000}%
\pgfsetfillcolor{currentfill}%
\pgfsetlinewidth{0.501875pt}%
\definecolor{currentstroke}{rgb}{0.000000,0.000000,0.000000}%
\pgfsetstrokecolor{currentstroke}%
\pgfsetdash{}{0pt}%
\pgfsys@defobject{currentmarker}{\pgfqpoint{0.000000in}{0.000000in}}{\pgfqpoint{0.000000in}{0.055556in}}{%
\pgfpathmoveto{\pgfqpoint{0.000000in}{0.000000in}}%
\pgfpathlineto{\pgfqpoint{0.000000in}{0.055556in}}%
\pgfusepath{stroke,fill}%
}%
\begin{pgfscope}%
\pgfsys@transformshift{6.166667in}{0.600000in}%
\pgfsys@useobject{currentmarker}{}%
\end{pgfscope}%
\end{pgfscope}%
\begin{pgfscope}%
\pgfsetbuttcap%
\pgfsetroundjoin%
\definecolor{currentfill}{rgb}{0.000000,0.000000,0.000000}%
\pgfsetfillcolor{currentfill}%
\pgfsetlinewidth{0.501875pt}%
\definecolor{currentstroke}{rgb}{0.000000,0.000000,0.000000}%
\pgfsetstrokecolor{currentstroke}%
\pgfsetdash{}{0pt}%
\pgfsys@defobject{currentmarker}{\pgfqpoint{0.000000in}{-0.055556in}}{\pgfqpoint{0.000000in}{0.000000in}}{%
\pgfpathmoveto{\pgfqpoint{0.000000in}{0.000000in}}%
\pgfpathlineto{\pgfqpoint{0.000000in}{-0.055556in}}%
\pgfusepath{stroke,fill}%
}%
\begin{pgfscope}%
\pgfsys@transformshift{6.166667in}{5.400000in}%
\pgfsys@useobject{currentmarker}{}%
\end{pgfscope}%
\end{pgfscope}%
\begin{pgfscope}%
\pgftext[x=6.166667in,y=0.544444in,,top]{\sffamily\fontsize{12.000000}{14.400000}\selectfont 5}%
\end{pgfscope}%
\begin{pgfscope}%
\pgfsetbuttcap%
\pgfsetroundjoin%
\definecolor{currentfill}{rgb}{0.000000,0.000000,0.000000}%
\pgfsetfillcolor{currentfill}%
\pgfsetlinewidth{0.501875pt}%
\definecolor{currentstroke}{rgb}{0.000000,0.000000,0.000000}%
\pgfsetstrokecolor{currentstroke}%
\pgfsetdash{}{0pt}%
\pgfsys@defobject{currentmarker}{\pgfqpoint{0.000000in}{0.000000in}}{\pgfqpoint{0.000000in}{0.055556in}}{%
\pgfpathmoveto{\pgfqpoint{0.000000in}{0.000000in}}%
\pgfpathlineto{\pgfqpoint{0.000000in}{0.055556in}}%
\pgfusepath{stroke,fill}%
}%
\begin{pgfscope}%
\pgfsys@transformshift{7.200000in}{0.600000in}%
\pgfsys@useobject{currentmarker}{}%
\end{pgfscope}%
\end{pgfscope}%
\begin{pgfscope}%
\pgfsetbuttcap%
\pgfsetroundjoin%
\definecolor{currentfill}{rgb}{0.000000,0.000000,0.000000}%
\pgfsetfillcolor{currentfill}%
\pgfsetlinewidth{0.501875pt}%
\definecolor{currentstroke}{rgb}{0.000000,0.000000,0.000000}%
\pgfsetstrokecolor{currentstroke}%
\pgfsetdash{}{0pt}%
\pgfsys@defobject{currentmarker}{\pgfqpoint{0.000000in}{-0.055556in}}{\pgfqpoint{0.000000in}{0.000000in}}{%
\pgfpathmoveto{\pgfqpoint{0.000000in}{0.000000in}}%
\pgfpathlineto{\pgfqpoint{0.000000in}{-0.055556in}}%
\pgfusepath{stroke,fill}%
}%
\begin{pgfscope}%
\pgfsys@transformshift{7.200000in}{5.400000in}%
\pgfsys@useobject{currentmarker}{}%
\end{pgfscope}%
\end{pgfscope}%
\begin{pgfscope}%
\pgftext[x=7.200000in,y=0.544444in,,top]{\sffamily\fontsize{12.000000}{14.400000}\selectfont 6}%
\end{pgfscope}%
\begin{pgfscope}%
\pgftext[x=4.100000in,y=0.313705in,,top]{\sffamily\fontsize{12.000000}{14.400000}\selectfont Refinement level}%
\end{pgfscope}%
\begin{pgfscope}%
\pgfsetbuttcap%
\pgfsetroundjoin%
\definecolor{currentfill}{rgb}{0.000000,0.000000,0.000000}%
\pgfsetfillcolor{currentfill}%
\pgfsetlinewidth{0.501875pt}%
\definecolor{currentstroke}{rgb}{0.000000,0.000000,0.000000}%
\pgfsetstrokecolor{currentstroke}%
\pgfsetdash{}{0pt}%
\pgfsys@defobject{currentmarker}{\pgfqpoint{0.000000in}{0.000000in}}{\pgfqpoint{0.055556in}{0.000000in}}{%
\pgfpathmoveto{\pgfqpoint{0.000000in}{0.000000in}}%
\pgfpathlineto{\pgfqpoint{0.055556in}{0.000000in}}%
\pgfusepath{stroke,fill}%
}%
\begin{pgfscope}%
\pgfsys@transformshift{1.000000in}{0.600000in}%
\pgfsys@useobject{currentmarker}{}%
\end{pgfscope}%
\end{pgfscope}%
\begin{pgfscope}%
\pgfsetbuttcap%
\pgfsetroundjoin%
\definecolor{currentfill}{rgb}{0.000000,0.000000,0.000000}%
\pgfsetfillcolor{currentfill}%
\pgfsetlinewidth{0.501875pt}%
\definecolor{currentstroke}{rgb}{0.000000,0.000000,0.000000}%
\pgfsetstrokecolor{currentstroke}%
\pgfsetdash{}{0pt}%
\pgfsys@defobject{currentmarker}{\pgfqpoint{-0.055556in}{0.000000in}}{\pgfqpoint{0.000000in}{0.000000in}}{%
\pgfpathmoveto{\pgfqpoint{0.000000in}{0.000000in}}%
\pgfpathlineto{\pgfqpoint{-0.055556in}{0.000000in}}%
\pgfusepath{stroke,fill}%
}%
\begin{pgfscope}%
\pgfsys@transformshift{7.200000in}{0.600000in}%
\pgfsys@useobject{currentmarker}{}%
\end{pgfscope}%
\end{pgfscope}%
\begin{pgfscope}%
\pgftext[x=0.944444in,y=0.600000in,right,]{\sffamily\fontsize{12.000000}{14.400000}\selectfont 0.00000}%
\end{pgfscope}%
\begin{pgfscope}%
\pgfsetbuttcap%
\pgfsetroundjoin%
\definecolor{currentfill}{rgb}{0.000000,0.000000,0.000000}%
\pgfsetfillcolor{currentfill}%
\pgfsetlinewidth{0.501875pt}%
\definecolor{currentstroke}{rgb}{0.000000,0.000000,0.000000}%
\pgfsetstrokecolor{currentstroke}%
\pgfsetdash{}{0pt}%
\pgfsys@defobject{currentmarker}{\pgfqpoint{0.000000in}{0.000000in}}{\pgfqpoint{0.055556in}{0.000000in}}{%
\pgfpathmoveto{\pgfqpoint{0.000000in}{0.000000in}}%
\pgfpathlineto{\pgfqpoint{0.055556in}{0.000000in}}%
\pgfusepath{stroke,fill}%
}%
\begin{pgfscope}%
\pgfsys@transformshift{1.000000in}{1.200000in}%
\pgfsys@useobject{currentmarker}{}%
\end{pgfscope}%
\end{pgfscope}%
\begin{pgfscope}%
\pgfsetbuttcap%
\pgfsetroundjoin%
\definecolor{currentfill}{rgb}{0.000000,0.000000,0.000000}%
\pgfsetfillcolor{currentfill}%
\pgfsetlinewidth{0.501875pt}%
\definecolor{currentstroke}{rgb}{0.000000,0.000000,0.000000}%
\pgfsetstrokecolor{currentstroke}%
\pgfsetdash{}{0pt}%
\pgfsys@defobject{currentmarker}{\pgfqpoint{-0.055556in}{0.000000in}}{\pgfqpoint{0.000000in}{0.000000in}}{%
\pgfpathmoveto{\pgfqpoint{0.000000in}{0.000000in}}%
\pgfpathlineto{\pgfqpoint{-0.055556in}{0.000000in}}%
\pgfusepath{stroke,fill}%
}%
\begin{pgfscope}%
\pgfsys@transformshift{7.200000in}{1.200000in}%
\pgfsys@useobject{currentmarker}{}%
\end{pgfscope}%
\end{pgfscope}%
\begin{pgfscope}%
\pgftext[x=0.944444in,y=1.200000in,right,]{\sffamily\fontsize{12.000000}{14.400000}\selectfont 0.00001}%
\end{pgfscope}%
\begin{pgfscope}%
\pgfsetbuttcap%
\pgfsetroundjoin%
\definecolor{currentfill}{rgb}{0.000000,0.000000,0.000000}%
\pgfsetfillcolor{currentfill}%
\pgfsetlinewidth{0.501875pt}%
\definecolor{currentstroke}{rgb}{0.000000,0.000000,0.000000}%
\pgfsetstrokecolor{currentstroke}%
\pgfsetdash{}{0pt}%
\pgfsys@defobject{currentmarker}{\pgfqpoint{0.000000in}{0.000000in}}{\pgfqpoint{0.055556in}{0.000000in}}{%
\pgfpathmoveto{\pgfqpoint{0.000000in}{0.000000in}}%
\pgfpathlineto{\pgfqpoint{0.055556in}{0.000000in}}%
\pgfusepath{stroke,fill}%
}%
\begin{pgfscope}%
\pgfsys@transformshift{1.000000in}{1.800000in}%
\pgfsys@useobject{currentmarker}{}%
\end{pgfscope}%
\end{pgfscope}%
\begin{pgfscope}%
\pgfsetbuttcap%
\pgfsetroundjoin%
\definecolor{currentfill}{rgb}{0.000000,0.000000,0.000000}%
\pgfsetfillcolor{currentfill}%
\pgfsetlinewidth{0.501875pt}%
\definecolor{currentstroke}{rgb}{0.000000,0.000000,0.000000}%
\pgfsetstrokecolor{currentstroke}%
\pgfsetdash{}{0pt}%
\pgfsys@defobject{currentmarker}{\pgfqpoint{-0.055556in}{0.000000in}}{\pgfqpoint{0.000000in}{0.000000in}}{%
\pgfpathmoveto{\pgfqpoint{0.000000in}{0.000000in}}%
\pgfpathlineto{\pgfqpoint{-0.055556in}{0.000000in}}%
\pgfusepath{stroke,fill}%
}%
\begin{pgfscope}%
\pgfsys@transformshift{7.200000in}{1.800000in}%
\pgfsys@useobject{currentmarker}{}%
\end{pgfscope}%
\end{pgfscope}%
\begin{pgfscope}%
\pgftext[x=0.944444in,y=1.800000in,right,]{\sffamily\fontsize{12.000000}{14.400000}\selectfont 0.00002}%
\end{pgfscope}%
\begin{pgfscope}%
\pgfsetbuttcap%
\pgfsetroundjoin%
\definecolor{currentfill}{rgb}{0.000000,0.000000,0.000000}%
\pgfsetfillcolor{currentfill}%
\pgfsetlinewidth{0.501875pt}%
\definecolor{currentstroke}{rgb}{0.000000,0.000000,0.000000}%
\pgfsetstrokecolor{currentstroke}%
\pgfsetdash{}{0pt}%
\pgfsys@defobject{currentmarker}{\pgfqpoint{0.000000in}{0.000000in}}{\pgfqpoint{0.055556in}{0.000000in}}{%
\pgfpathmoveto{\pgfqpoint{0.000000in}{0.000000in}}%
\pgfpathlineto{\pgfqpoint{0.055556in}{0.000000in}}%
\pgfusepath{stroke,fill}%
}%
\begin{pgfscope}%
\pgfsys@transformshift{1.000000in}{2.400000in}%
\pgfsys@useobject{currentmarker}{}%
\end{pgfscope}%
\end{pgfscope}%
\begin{pgfscope}%
\pgfsetbuttcap%
\pgfsetroundjoin%
\definecolor{currentfill}{rgb}{0.000000,0.000000,0.000000}%
\pgfsetfillcolor{currentfill}%
\pgfsetlinewidth{0.501875pt}%
\definecolor{currentstroke}{rgb}{0.000000,0.000000,0.000000}%
\pgfsetstrokecolor{currentstroke}%
\pgfsetdash{}{0pt}%
\pgfsys@defobject{currentmarker}{\pgfqpoint{-0.055556in}{0.000000in}}{\pgfqpoint{0.000000in}{0.000000in}}{%
\pgfpathmoveto{\pgfqpoint{0.000000in}{0.000000in}}%
\pgfpathlineto{\pgfqpoint{-0.055556in}{0.000000in}}%
\pgfusepath{stroke,fill}%
}%
\begin{pgfscope}%
\pgfsys@transformshift{7.200000in}{2.400000in}%
\pgfsys@useobject{currentmarker}{}%
\end{pgfscope}%
\end{pgfscope}%
\begin{pgfscope}%
\pgftext[x=0.944444in,y=2.400000in,right,]{\sffamily\fontsize{12.000000}{14.400000}\selectfont 0.00003}%
\end{pgfscope}%
\begin{pgfscope}%
\pgfsetbuttcap%
\pgfsetroundjoin%
\definecolor{currentfill}{rgb}{0.000000,0.000000,0.000000}%
\pgfsetfillcolor{currentfill}%
\pgfsetlinewidth{0.501875pt}%
\definecolor{currentstroke}{rgb}{0.000000,0.000000,0.000000}%
\pgfsetstrokecolor{currentstroke}%
\pgfsetdash{}{0pt}%
\pgfsys@defobject{currentmarker}{\pgfqpoint{0.000000in}{0.000000in}}{\pgfqpoint{0.055556in}{0.000000in}}{%
\pgfpathmoveto{\pgfqpoint{0.000000in}{0.000000in}}%
\pgfpathlineto{\pgfqpoint{0.055556in}{0.000000in}}%
\pgfusepath{stroke,fill}%
}%
\begin{pgfscope}%
\pgfsys@transformshift{1.000000in}{3.000000in}%
\pgfsys@useobject{currentmarker}{}%
\end{pgfscope}%
\end{pgfscope}%
\begin{pgfscope}%
\pgfsetbuttcap%
\pgfsetroundjoin%
\definecolor{currentfill}{rgb}{0.000000,0.000000,0.000000}%
\pgfsetfillcolor{currentfill}%
\pgfsetlinewidth{0.501875pt}%
\definecolor{currentstroke}{rgb}{0.000000,0.000000,0.000000}%
\pgfsetstrokecolor{currentstroke}%
\pgfsetdash{}{0pt}%
\pgfsys@defobject{currentmarker}{\pgfqpoint{-0.055556in}{0.000000in}}{\pgfqpoint{0.000000in}{0.000000in}}{%
\pgfpathmoveto{\pgfqpoint{0.000000in}{0.000000in}}%
\pgfpathlineto{\pgfqpoint{-0.055556in}{0.000000in}}%
\pgfusepath{stroke,fill}%
}%
\begin{pgfscope}%
\pgfsys@transformshift{7.200000in}{3.000000in}%
\pgfsys@useobject{currentmarker}{}%
\end{pgfscope}%
\end{pgfscope}%
\begin{pgfscope}%
\pgftext[x=0.944444in,y=3.000000in,right,]{\sffamily\fontsize{12.000000}{14.400000}\selectfont 0.00004}%
\end{pgfscope}%
\begin{pgfscope}%
\pgfsetbuttcap%
\pgfsetroundjoin%
\definecolor{currentfill}{rgb}{0.000000,0.000000,0.000000}%
\pgfsetfillcolor{currentfill}%
\pgfsetlinewidth{0.501875pt}%
\definecolor{currentstroke}{rgb}{0.000000,0.000000,0.000000}%
\pgfsetstrokecolor{currentstroke}%
\pgfsetdash{}{0pt}%
\pgfsys@defobject{currentmarker}{\pgfqpoint{0.000000in}{0.000000in}}{\pgfqpoint{0.055556in}{0.000000in}}{%
\pgfpathmoveto{\pgfqpoint{0.000000in}{0.000000in}}%
\pgfpathlineto{\pgfqpoint{0.055556in}{0.000000in}}%
\pgfusepath{stroke,fill}%
}%
\begin{pgfscope}%
\pgfsys@transformshift{1.000000in}{3.600000in}%
\pgfsys@useobject{currentmarker}{}%
\end{pgfscope}%
\end{pgfscope}%
\begin{pgfscope}%
\pgfsetbuttcap%
\pgfsetroundjoin%
\definecolor{currentfill}{rgb}{0.000000,0.000000,0.000000}%
\pgfsetfillcolor{currentfill}%
\pgfsetlinewidth{0.501875pt}%
\definecolor{currentstroke}{rgb}{0.000000,0.000000,0.000000}%
\pgfsetstrokecolor{currentstroke}%
\pgfsetdash{}{0pt}%
\pgfsys@defobject{currentmarker}{\pgfqpoint{-0.055556in}{0.000000in}}{\pgfqpoint{0.000000in}{0.000000in}}{%
\pgfpathmoveto{\pgfqpoint{0.000000in}{0.000000in}}%
\pgfpathlineto{\pgfqpoint{-0.055556in}{0.000000in}}%
\pgfusepath{stroke,fill}%
}%
\begin{pgfscope}%
\pgfsys@transformshift{7.200000in}{3.600000in}%
\pgfsys@useobject{currentmarker}{}%
\end{pgfscope}%
\end{pgfscope}%
\begin{pgfscope}%
\pgftext[x=0.944444in,y=3.600000in,right,]{\sffamily\fontsize{12.000000}{14.400000}\selectfont 0.00005}%
\end{pgfscope}%
\begin{pgfscope}%
\pgfsetbuttcap%
\pgfsetroundjoin%
\definecolor{currentfill}{rgb}{0.000000,0.000000,0.000000}%
\pgfsetfillcolor{currentfill}%
\pgfsetlinewidth{0.501875pt}%
\definecolor{currentstroke}{rgb}{0.000000,0.000000,0.000000}%
\pgfsetstrokecolor{currentstroke}%
\pgfsetdash{}{0pt}%
\pgfsys@defobject{currentmarker}{\pgfqpoint{0.000000in}{0.000000in}}{\pgfqpoint{0.055556in}{0.000000in}}{%
\pgfpathmoveto{\pgfqpoint{0.000000in}{0.000000in}}%
\pgfpathlineto{\pgfqpoint{0.055556in}{0.000000in}}%
\pgfusepath{stroke,fill}%
}%
\begin{pgfscope}%
\pgfsys@transformshift{1.000000in}{4.200000in}%
\pgfsys@useobject{currentmarker}{}%
\end{pgfscope}%
\end{pgfscope}%
\begin{pgfscope}%
\pgfsetbuttcap%
\pgfsetroundjoin%
\definecolor{currentfill}{rgb}{0.000000,0.000000,0.000000}%
\pgfsetfillcolor{currentfill}%
\pgfsetlinewidth{0.501875pt}%
\definecolor{currentstroke}{rgb}{0.000000,0.000000,0.000000}%
\pgfsetstrokecolor{currentstroke}%
\pgfsetdash{}{0pt}%
\pgfsys@defobject{currentmarker}{\pgfqpoint{-0.055556in}{0.000000in}}{\pgfqpoint{0.000000in}{0.000000in}}{%
\pgfpathmoveto{\pgfqpoint{0.000000in}{0.000000in}}%
\pgfpathlineto{\pgfqpoint{-0.055556in}{0.000000in}}%
\pgfusepath{stroke,fill}%
}%
\begin{pgfscope}%
\pgfsys@transformshift{7.200000in}{4.200000in}%
\pgfsys@useobject{currentmarker}{}%
\end{pgfscope}%
\end{pgfscope}%
\begin{pgfscope}%
\pgftext[x=0.944444in,y=4.200000in,right,]{\sffamily\fontsize{12.000000}{14.400000}\selectfont 0.00006}%
\end{pgfscope}%
\begin{pgfscope}%
\pgfsetbuttcap%
\pgfsetroundjoin%
\definecolor{currentfill}{rgb}{0.000000,0.000000,0.000000}%
\pgfsetfillcolor{currentfill}%
\pgfsetlinewidth{0.501875pt}%
\definecolor{currentstroke}{rgb}{0.000000,0.000000,0.000000}%
\pgfsetstrokecolor{currentstroke}%
\pgfsetdash{}{0pt}%
\pgfsys@defobject{currentmarker}{\pgfqpoint{0.000000in}{0.000000in}}{\pgfqpoint{0.055556in}{0.000000in}}{%
\pgfpathmoveto{\pgfqpoint{0.000000in}{0.000000in}}%
\pgfpathlineto{\pgfqpoint{0.055556in}{0.000000in}}%
\pgfusepath{stroke,fill}%
}%
\begin{pgfscope}%
\pgfsys@transformshift{1.000000in}{4.800000in}%
\pgfsys@useobject{currentmarker}{}%
\end{pgfscope}%
\end{pgfscope}%
\begin{pgfscope}%
\pgfsetbuttcap%
\pgfsetroundjoin%
\definecolor{currentfill}{rgb}{0.000000,0.000000,0.000000}%
\pgfsetfillcolor{currentfill}%
\pgfsetlinewidth{0.501875pt}%
\definecolor{currentstroke}{rgb}{0.000000,0.000000,0.000000}%
\pgfsetstrokecolor{currentstroke}%
\pgfsetdash{}{0pt}%
\pgfsys@defobject{currentmarker}{\pgfqpoint{-0.055556in}{0.000000in}}{\pgfqpoint{0.000000in}{0.000000in}}{%
\pgfpathmoveto{\pgfqpoint{0.000000in}{0.000000in}}%
\pgfpathlineto{\pgfqpoint{-0.055556in}{0.000000in}}%
\pgfusepath{stroke,fill}%
}%
\begin{pgfscope}%
\pgfsys@transformshift{7.200000in}{4.800000in}%
\pgfsys@useobject{currentmarker}{}%
\end{pgfscope}%
\end{pgfscope}%
\begin{pgfscope}%
\pgftext[x=0.944444in,y=4.800000in,right,]{\sffamily\fontsize{12.000000}{14.400000}\selectfont 0.00007}%
\end{pgfscope}%
\begin{pgfscope}%
\pgfsetbuttcap%
\pgfsetroundjoin%
\definecolor{currentfill}{rgb}{0.000000,0.000000,0.000000}%
\pgfsetfillcolor{currentfill}%
\pgfsetlinewidth{0.501875pt}%
\definecolor{currentstroke}{rgb}{0.000000,0.000000,0.000000}%
\pgfsetstrokecolor{currentstroke}%
\pgfsetdash{}{0pt}%
\pgfsys@defobject{currentmarker}{\pgfqpoint{0.000000in}{0.000000in}}{\pgfqpoint{0.055556in}{0.000000in}}{%
\pgfpathmoveto{\pgfqpoint{0.000000in}{0.000000in}}%
\pgfpathlineto{\pgfqpoint{0.055556in}{0.000000in}}%
\pgfusepath{stroke,fill}%
}%
\begin{pgfscope}%
\pgfsys@transformshift{1.000000in}{5.400000in}%
\pgfsys@useobject{currentmarker}{}%
\end{pgfscope}%
\end{pgfscope}%
\begin{pgfscope}%
\pgfsetbuttcap%
\pgfsetroundjoin%
\definecolor{currentfill}{rgb}{0.000000,0.000000,0.000000}%
\pgfsetfillcolor{currentfill}%
\pgfsetlinewidth{0.501875pt}%
\definecolor{currentstroke}{rgb}{0.000000,0.000000,0.000000}%
\pgfsetstrokecolor{currentstroke}%
\pgfsetdash{}{0pt}%
\pgfsys@defobject{currentmarker}{\pgfqpoint{-0.055556in}{0.000000in}}{\pgfqpoint{0.000000in}{0.000000in}}{%
\pgfpathmoveto{\pgfqpoint{0.000000in}{0.000000in}}%
\pgfpathlineto{\pgfqpoint{-0.055556in}{0.000000in}}%
\pgfusepath{stroke,fill}%
}%
\begin{pgfscope}%
\pgfsys@transformshift{7.200000in}{5.400000in}%
\pgfsys@useobject{currentmarker}{}%
\end{pgfscope}%
\end{pgfscope}%
\begin{pgfscope}%
\pgftext[x=0.944444in,y=5.400000in,right,]{\sffamily\fontsize{12.000000}{14.400000}\selectfont 0.00008}%
\end{pgfscope}%
\begin{pgfscope}%
\pgftext[x=0.185791in,y=3.000000in,,bottom,rotate=90.000000]{\sffamily\fontsize{12.000000}{14.400000}\selectfont Functional}%
\end{pgfscope}%
\begin{pgfscope}%
\pgfsetbuttcap%
\pgfsetmiterjoin%
\definecolor{currentfill}{rgb}{1.000000,1.000000,1.000000}%
\pgfsetfillcolor{currentfill}%
\pgfsetlinewidth{1.003750pt}%
\definecolor{currentstroke}{rgb}{0.000000,0.000000,0.000000}%
\pgfsetstrokecolor{currentstroke}%
\pgfsetdash{}{0pt}%
\pgfpathmoveto{\pgfqpoint{5.662617in}{0.700000in}}%
\pgfpathlineto{\pgfqpoint{7.100000in}{0.700000in}}%
\pgfpathlineto{\pgfqpoint{7.100000in}{1.640664in}}%
\pgfpathlineto{\pgfqpoint{5.662617in}{1.640664in}}%
\pgfpathclose%
\pgfusepath{stroke,fill}%
\end{pgfscope}%
\begin{pgfscope}%
\pgfsetbuttcap%
\pgfsetroundjoin%
\pgfsetlinewidth{1.003750pt}%
\definecolor{currentstroke}{rgb}{0.000000,0.000000,0.000000}%
\pgfsetstrokecolor{currentstroke}%
\pgfsetdash{{6.000000pt}{6.000000pt}}{0.000000pt}%
\pgfpathmoveto{\pgfqpoint{5.802617in}{1.478711in}}%
\pgfpathlineto{\pgfqpoint{6.082617in}{1.478711in}}%
\pgfusepath{stroke}%
\end{pgfscope}%
\begin{pgfscope}%
\pgftext[x=6.302617in,y=1.408711in,left,base]{\sffamily\fontsize{14.400000}{17.280000}\selectfont Mesh 1}%
\end{pgfscope}%
\begin{pgfscope}%
\pgfsetbuttcap%
\pgfsetroundjoin%
\pgfsetlinewidth{1.003750pt}%
\definecolor{currentstroke}{rgb}{0.000000,0.000000,0.000000}%
\pgfsetstrokecolor{currentstroke}%
\pgfsetdash{{3.000000pt}{5.000000pt}{1.000000pt}{5.000000pt}}{0.000000pt}%
\pgfpathmoveto{\pgfqpoint{5.802617in}{1.185156in}}%
\pgfpathlineto{\pgfqpoint{6.082617in}{1.185156in}}%
\pgfusepath{stroke}%
\end{pgfscope}%
\begin{pgfscope}%
\pgftext[x=6.302617in,y=1.115156in,left,base]{\sffamily\fontsize{14.400000}{17.280000}\selectfont Mesh 2}%
\end{pgfscope}%
\begin{pgfscope}%
\pgfsetrectcap%
\pgfsetroundjoin%
\pgfsetlinewidth{1.003750pt}%
\definecolor{currentstroke}{rgb}{0.000000,0.000000,0.000000}%
\pgfsetstrokecolor{currentstroke}%
\pgfsetdash{}{0pt}%
\pgfpathmoveto{\pgfqpoint{5.802617in}{0.891601in}}%
\pgfpathlineto{\pgfqpoint{6.082617in}{0.891601in}}%
\pgfusepath{stroke}%
\end{pgfscope}%
\begin{pgfscope}%
\pgftext[x=6.302617in,y=0.821601in,left,base]{\sffamily\fontsize{14.400000}{17.280000}\selectfont Mesh 3}%
\end{pgfscope}%
\end{pgfpicture}%
\makeatother%
\endgroup%

%% file: Gao_multilevel.pgf
\begingroup%
\makeatletter%
\begin{pgfpicture}%
\pgfpathrectangle{\pgfpointorigin}{\pgfqpoint{8.000000in}{6.000000in}}%
\pgfusepath{use as bounding box, clip}%
\begin{pgfscope}%
\pgfsetbuttcap%
\pgfsetmiterjoin%
\definecolor{currentfill}{rgb}{1.000000,1.000000,1.000000}%
\pgfsetfillcolor{currentfill}%
\pgfsetlinewidth{0.000000pt}%
\definecolor{currentstroke}{rgb}{1.000000,1.000000,1.000000}%
\pgfsetstrokecolor{currentstroke}%
\pgfsetdash{}{0pt}%
\pgfpathmoveto{\pgfqpoint{0.000000in}{0.000000in}}%
\pgfpathlineto{\pgfqpoint{8.000000in}{0.000000in}}%
\pgfpathlineto{\pgfqpoint{8.000000in}{6.000000in}}%
\pgfpathlineto{\pgfqpoint{0.000000in}{6.000000in}}%
\pgfpathclose%
\pgfusepath{fill}%
\end{pgfscope}%
\begin{pgfscope}%
\pgfsetbuttcap%
\pgfsetmiterjoin%
\definecolor{currentfill}{rgb}{1.000000,1.000000,1.000000}%
\pgfsetfillcolor{currentfill}%
\pgfsetlinewidth{0.000000pt}%
\definecolor{currentstroke}{rgb}{0.000000,0.000000,0.000000}%
\pgfsetstrokecolor{currentstroke}%
\pgfsetstrokeopacity{0.000000}%
\pgfsetdash{}{0pt}%
\pgfpathmoveto{\pgfqpoint{1.000000in}{0.600000in}}%
\pgfpathlineto{\pgfqpoint{7.200000in}{0.600000in}}%
\pgfpathlineto{\pgfqpoint{7.200000in}{5.400000in}}%
\pgfpathlineto{\pgfqpoint{1.000000in}{5.400000in}}%
\pgfpathclose%
\pgfusepath{fill}%
\end{pgfscope}%
\begin{pgfscope}%
\pgfpathrectangle{\pgfqpoint{1.000000in}{0.600000in}}{\pgfqpoint{6.200000in}{4.800000in}} %
\pgfusepath{clip}%
\pgfsetbuttcap%
\pgfsetroundjoin%
\pgfsetlinewidth{1.003750pt}%
\definecolor{currentstroke}{rgb}{0.000000,0.000000,0.000000}%
\pgfsetstrokecolor{currentstroke}%
\pgfsetdash{{6.000000pt}{6.000000pt}}{0.000000pt}%
\pgfpathmoveto{\pgfqpoint{1.000000in}{0.652391in}}%
\pgfpathlineto{\pgfqpoint{2.033333in}{4.760379in}}%
\pgfpathlineto{\pgfqpoint{3.066667in}{5.038223in}}%
\pgfpathlineto{\pgfqpoint{4.100000in}{5.159857in}}%
\pgfpathlineto{\pgfqpoint{5.133333in}{5.223929in}}%
\pgfpathlineto{\pgfqpoint{6.166667in}{5.266087in}}%
\pgfpathlineto{\pgfqpoint{7.200000in}{5.299361in}}%
\pgfusepath{stroke}%
\end{pgfscope}%
\begin{pgfscope}%
\pgfpathrectangle{\pgfqpoint{1.000000in}{0.600000in}}{\pgfqpoint{6.200000in}{4.800000in}} %
\pgfusepath{clip}%
\pgfsetbuttcap%
\pgfsetroundjoin%
\pgfsetlinewidth{1.003750pt}%
\definecolor{currentstroke}{rgb}{0.000000,0.000000,0.000000}%
\pgfsetstrokecolor{currentstroke}%
\pgfsetdash{{3.000000pt}{5.000000pt}{1.000000pt}{5.000000pt}}{0.000000pt}%
\pgfpathmoveto{\pgfqpoint{1.000000in}{0.639309in}}%
\pgfpathlineto{\pgfqpoint{2.033333in}{4.304390in}}%
\pgfpathlineto{\pgfqpoint{3.066667in}{4.537826in}}%
\pgfpathlineto{\pgfqpoint{4.100000in}{4.642112in}}%
\pgfpathlineto{\pgfqpoint{5.133333in}{4.698407in}}%
\pgfpathlineto{\pgfqpoint{6.166667in}{4.736255in}}%
\pgfpathlineto{\pgfqpoint{7.200000in}{4.766508in}}%
\pgfusepath{stroke}%
\end{pgfscope}%
\begin{pgfscope}%
\pgfpathrectangle{\pgfqpoint{1.000000in}{0.600000in}}{\pgfqpoint{6.200000in}{4.800000in}} %
\pgfusepath{clip}%
\pgfsetrectcap%
\pgfsetroundjoin%
\pgfsetlinewidth{1.003750pt}%
\definecolor{currentstroke}{rgb}{0.000000,0.000000,0.000000}%
\pgfsetstrokecolor{currentstroke}%
\pgfsetdash{}{0pt}%
\pgfpathmoveto{\pgfqpoint{1.000000in}{0.637503in}}%
\pgfpathlineto{\pgfqpoint{2.033333in}{4.151290in}}%
\pgfpathlineto{\pgfqpoint{3.066667in}{4.378039in}}%
\pgfpathlineto{\pgfqpoint{4.100000in}{4.477643in}}%
\pgfpathlineto{\pgfqpoint{5.133333in}{4.531137in}}%
\pgfpathlineto{\pgfqpoint{6.166667in}{4.567313in}}%
\pgfpathlineto{\pgfqpoint{7.200000in}{4.596507in}}%
\pgfusepath{stroke}%
\end{pgfscope}%
\begin{pgfscope}%
\pgfsetrectcap%
\pgfsetmiterjoin%
\pgfsetlinewidth{1.003750pt}%
\definecolor{currentstroke}{rgb}{0.000000,0.000000,0.000000}%
\pgfsetstrokecolor{currentstroke}%
\pgfsetdash{}{0pt}%
\pgfpathmoveto{\pgfqpoint{1.000000in}{5.400000in}}%
\pgfpathlineto{\pgfqpoint{7.200000in}{5.400000in}}%
\pgfusepath{stroke}%
\end{pgfscope}%
\begin{pgfscope}%
\pgfsetrectcap%
\pgfsetmiterjoin%
\pgfsetlinewidth{1.003750pt}%
\definecolor{currentstroke}{rgb}{0.000000,0.000000,0.000000}%
\pgfsetstrokecolor{currentstroke}%
\pgfsetdash{}{0pt}%
\pgfpathmoveto{\pgfqpoint{7.200000in}{0.600000in}}%
\pgfpathlineto{\pgfqpoint{7.200000in}{5.400000in}}%
\pgfusepath{stroke}%
\end{pgfscope}%
\begin{pgfscope}%
\pgfsetrectcap%
\pgfsetmiterjoin%
\pgfsetlinewidth{1.003750pt}%
\definecolor{currentstroke}{rgb}{0.000000,0.000000,0.000000}%
\pgfsetstrokecolor{currentstroke}%
\pgfsetdash{}{0pt}%
\pgfpathmoveto{\pgfqpoint{1.000000in}{0.600000in}}%
\pgfpathlineto{\pgfqpoint{7.200000in}{0.600000in}}%
\pgfusepath{stroke}%
\end{pgfscope}%
\begin{pgfscope}%
\pgfsetrectcap%
\pgfsetmiterjoin%
\pgfsetlinewidth{1.003750pt}%
\definecolor{currentstroke}{rgb}{0.000000,0.000000,0.000000}%
\pgfsetstrokecolor{currentstroke}%
\pgfsetdash{}{0pt}%
\pgfpathmoveto{\pgfqpoint{1.000000in}{0.600000in}}%
\pgfpathlineto{\pgfqpoint{1.000000in}{5.400000in}}%
\pgfusepath{stroke}%
\end{pgfscope}%
\begin{pgfscope}%
\pgfsetbuttcap%
\pgfsetroundjoin%
\definecolor{currentfill}{rgb}{0.000000,0.000000,0.000000}%
\pgfsetfillcolor{currentfill}%
\pgfsetlinewidth{0.501875pt}%
\definecolor{currentstroke}{rgb}{0.000000,0.000000,0.000000}%
\pgfsetstrokecolor{currentstroke}%
\pgfsetdash{}{0pt}%
\pgfsys@defobject{currentmarker}{\pgfqpoint{0.000000in}{0.000000in}}{\pgfqpoint{0.000000in}{0.055556in}}{%
\pgfpathmoveto{\pgfqpoint{0.000000in}{0.000000in}}%
\pgfpathlineto{\pgfqpoint{0.000000in}{0.055556in}}%
\pgfusepath{stroke,fill}%
}%
\begin{pgfscope}%
\pgfsys@transformshift{1.000000in}{0.600000in}%
\pgfsys@useobject{currentmarker}{}%
\end{pgfscope}%
\end{pgfscope}%
\begin{pgfscope}%
\pgfsetbuttcap%
\pgfsetroundjoin%
\definecolor{currentfill}{rgb}{0.000000,0.000000,0.000000}%
\pgfsetfillcolor{currentfill}%
\pgfsetlinewidth{0.501875pt}%
\definecolor{currentstroke}{rgb}{0.000000,0.000000,0.000000}%
\pgfsetstrokecolor{currentstroke}%
\pgfsetdash{}{0pt}%
\pgfsys@defobject{currentmarker}{\pgfqpoint{0.000000in}{-0.055556in}}{\pgfqpoint{0.000000in}{0.000000in}}{%
\pgfpathmoveto{\pgfqpoint{0.000000in}{0.000000in}}%
\pgfpathlineto{\pgfqpoint{0.000000in}{-0.055556in}}%
\pgfusepath{stroke,fill}%
}%
\begin{pgfscope}%
\pgfsys@transformshift{1.000000in}{5.400000in}%
\pgfsys@useobject{currentmarker}{}%
\end{pgfscope}%
\end{pgfscope}%
\begin{pgfscope}%
\pgftext[x=1.000000in,y=0.544444in,,top]{\sffamily\fontsize{12.000000}{14.400000}\selectfont 0}%
\end{pgfscope}%
\begin{pgfscope}%
\pgfsetbuttcap%
\pgfsetroundjoin%
\definecolor{currentfill}{rgb}{0.000000,0.000000,0.000000}%
\pgfsetfillcolor{currentfill}%
\pgfsetlinewidth{0.501875pt}%
\definecolor{currentstroke}{rgb}{0.000000,0.000000,0.000000}%
\pgfsetstrokecolor{currentstroke}%
\pgfsetdash{}{0pt}%
\pgfsys@defobject{currentmarker}{\pgfqpoint{0.000000in}{0.000000in}}{\pgfqpoint{0.000000in}{0.055556in}}{%
\pgfpathmoveto{\pgfqpoint{0.000000in}{0.000000in}}%
\pgfpathlineto{\pgfqpoint{0.000000in}{0.055556in}}%
\pgfusepath{stroke,fill}%
}%
\begin{pgfscope}%
\pgfsys@transformshift{2.033333in}{0.600000in}%
\pgfsys@useobject{currentmarker}{}%
\end{pgfscope}%
\end{pgfscope}%
\begin{pgfscope}%
\pgfsetbuttcap%
\pgfsetroundjoin%
\definecolor{currentfill}{rgb}{0.000000,0.000000,0.000000}%
\pgfsetfillcolor{currentfill}%
\pgfsetlinewidth{0.501875pt}%
\definecolor{currentstroke}{rgb}{0.000000,0.000000,0.000000}%
\pgfsetstrokecolor{currentstroke}%
\pgfsetdash{}{0pt}%
\pgfsys@defobject{currentmarker}{\pgfqpoint{0.000000in}{-0.055556in}}{\pgfqpoint{0.000000in}{0.000000in}}{%
\pgfpathmoveto{\pgfqpoint{0.000000in}{0.000000in}}%
\pgfpathlineto{\pgfqpoint{0.000000in}{-0.055556in}}%
\pgfusepath{stroke,fill}%
}%
\begin{pgfscope}%
\pgfsys@transformshift{2.033333in}{5.400000in}%
\pgfsys@useobject{currentmarker}{}%
\end{pgfscope}%
\end{pgfscope}%
\begin{pgfscope}%
\pgftext[x=2.033333in,y=0.544444in,,top]{\sffamily\fontsize{12.000000}{14.400000}\selectfont 1}%
\end{pgfscope}%
\begin{pgfscope}%
\pgfsetbuttcap%
\pgfsetroundjoin%
\definecolor{currentfill}{rgb}{0.000000,0.000000,0.000000}%
\pgfsetfillcolor{currentfill}%
\pgfsetlinewidth{0.501875pt}%
\definecolor{currentstroke}{rgb}{0.000000,0.000000,0.000000}%
\pgfsetstrokecolor{currentstroke}%
\pgfsetdash{}{0pt}%
\pgfsys@defobject{currentmarker}{\pgfqpoint{0.000000in}{0.000000in}}{\pgfqpoint{0.000000in}{0.055556in}}{%
\pgfpathmoveto{\pgfqpoint{0.000000in}{0.000000in}}%
\pgfpathlineto{\pgfqpoint{0.000000in}{0.055556in}}%
\pgfusepath{stroke,fill}%
}%
\begin{pgfscope}%
\pgfsys@transformshift{3.066667in}{0.600000in}%
\pgfsys@useobject{currentmarker}{}%
\end{pgfscope}%
\end{pgfscope}%
\begin{pgfscope}%
\pgfsetbuttcap%
\pgfsetroundjoin%
\definecolor{currentfill}{rgb}{0.000000,0.000000,0.000000}%
\pgfsetfillcolor{currentfill}%
\pgfsetlinewidth{0.501875pt}%
\definecolor{currentstroke}{rgb}{0.000000,0.000000,0.000000}%
\pgfsetstrokecolor{currentstroke}%
\pgfsetdash{}{0pt}%
\pgfsys@defobject{currentmarker}{\pgfqpoint{0.000000in}{-0.055556in}}{\pgfqpoint{0.000000in}{0.000000in}}{%
\pgfpathmoveto{\pgfqpoint{0.000000in}{0.000000in}}%
\pgfpathlineto{\pgfqpoint{0.000000in}{-0.055556in}}%
\pgfusepath{stroke,fill}%
}%
\begin{pgfscope}%
\pgfsys@transformshift{3.066667in}{5.400000in}%
\pgfsys@useobject{currentmarker}{}%
\end{pgfscope}%
\end{pgfscope}%
\begin{pgfscope}%
\pgftext[x=3.066667in,y=0.544444in,,top]{\sffamily\fontsize{12.000000}{14.400000}\selectfont 2}%
\end{pgfscope}%
\begin{pgfscope}%
\pgfsetbuttcap%
\pgfsetroundjoin%
\definecolor{currentfill}{rgb}{0.000000,0.000000,0.000000}%
\pgfsetfillcolor{currentfill}%
\pgfsetlinewidth{0.501875pt}%
\definecolor{currentstroke}{rgb}{0.000000,0.000000,0.000000}%
\pgfsetstrokecolor{currentstroke}%
\pgfsetdash{}{0pt}%
\pgfsys@defobject{currentmarker}{\pgfqpoint{0.000000in}{0.000000in}}{\pgfqpoint{0.000000in}{0.055556in}}{%
\pgfpathmoveto{\pgfqpoint{0.000000in}{0.000000in}}%
\pgfpathlineto{\pgfqpoint{0.000000in}{0.055556in}}%
\pgfusepath{stroke,fill}%
}%
\begin{pgfscope}%
\pgfsys@transformshift{4.100000in}{0.600000in}%
\pgfsys@useobject{currentmarker}{}%
\end{pgfscope}%
\end{pgfscope}%
\begin{pgfscope}%
\pgfsetbuttcap%
\pgfsetroundjoin%
\definecolor{currentfill}{rgb}{0.000000,0.000000,0.000000}%
\pgfsetfillcolor{currentfill}%
\pgfsetlinewidth{0.501875pt}%
\definecolor{currentstroke}{rgb}{0.000000,0.000000,0.000000}%
\pgfsetstrokecolor{currentstroke}%
\pgfsetdash{}{0pt}%
\pgfsys@defobject{currentmarker}{\pgfqpoint{0.000000in}{-0.055556in}}{\pgfqpoint{0.000000in}{0.000000in}}{%
\pgfpathmoveto{\pgfqpoint{0.000000in}{0.000000in}}%
\pgfpathlineto{\pgfqpoint{0.000000in}{-0.055556in}}%
\pgfusepath{stroke,fill}%
}%
\begin{pgfscope}%
\pgfsys@transformshift{4.100000in}{5.400000in}%
\pgfsys@useobject{currentmarker}{}%
\end{pgfscope}%
\end{pgfscope}%
\begin{pgfscope}%
\pgftext[x=4.100000in,y=0.544444in,,top]{\sffamily\fontsize{12.000000}{14.400000}\selectfont 3}%
\end{pgfscope}%
\begin{pgfscope}%
\pgfsetbuttcap%
\pgfsetroundjoin%
\definecolor{currentfill}{rgb}{0.000000,0.000000,0.000000}%
\pgfsetfillcolor{currentfill}%
\pgfsetlinewidth{0.501875pt}%
\definecolor{currentstroke}{rgb}{0.000000,0.000000,0.000000}%
\pgfsetstrokecolor{currentstroke}%
\pgfsetdash{}{0pt}%
\pgfsys@defobject{currentmarker}{\pgfqpoint{0.000000in}{0.000000in}}{\pgfqpoint{0.000000in}{0.055556in}}{%
\pgfpathmoveto{\pgfqpoint{0.000000in}{0.000000in}}%
\pgfpathlineto{\pgfqpoint{0.000000in}{0.055556in}}%
\pgfusepath{stroke,fill}%
}%
\begin{pgfscope}%
\pgfsys@transformshift{5.133333in}{0.600000in}%
\pgfsys@useobject{currentmarker}{}%
\end{pgfscope}%
\end{pgfscope}%
\begin{pgfscope}%
\pgfsetbuttcap%
\pgfsetroundjoin%
\definecolor{currentfill}{rgb}{0.000000,0.000000,0.000000}%
\pgfsetfillcolor{currentfill}%
\pgfsetlinewidth{0.501875pt}%
\definecolor{currentstroke}{rgb}{0.000000,0.000000,0.000000}%
\pgfsetstrokecolor{currentstroke}%
\pgfsetdash{}{0pt}%
\pgfsys@defobject{currentmarker}{\pgfqpoint{0.000000in}{-0.055556in}}{\pgfqpoint{0.000000in}{0.000000in}}{%
\pgfpathmoveto{\pgfqpoint{0.000000in}{0.000000in}}%
\pgfpathlineto{\pgfqpoint{0.000000in}{-0.055556in}}%
\pgfusepath{stroke,fill}%
}%
\begin{pgfscope}%
\pgfsys@transformshift{5.133333in}{5.400000in}%
\pgfsys@useobject{currentmarker}{}%
\end{pgfscope}%
\end{pgfscope}%
\begin{pgfscope}%
\pgftext[x=5.133333in,y=0.544444in,,top]{\sffamily\fontsize{12.000000}{14.400000}\selectfont 4}%
\end{pgfscope}%
\begin{pgfscope}%
\pgfsetbuttcap%
\pgfsetroundjoin%
\definecolor{currentfill}{rgb}{0.000000,0.000000,0.000000}%
\pgfsetfillcolor{currentfill}%
\pgfsetlinewidth{0.501875pt}%
\definecolor{currentstroke}{rgb}{0.000000,0.000000,0.000000}%
\pgfsetstrokecolor{currentstroke}%
\pgfsetdash{}{0pt}%
\pgfsys@defobject{currentmarker}{\pgfqpoint{0.000000in}{0.000000in}}{\pgfqpoint{0.000000in}{0.055556in}}{%
\pgfpathmoveto{\pgfqpoint{0.000000in}{0.000000in}}%
\pgfpathlineto{\pgfqpoint{0.000000in}{0.055556in}}%
\pgfusepath{stroke,fill}%
}%
\begin{pgfscope}%
\pgfsys@transformshift{6.166667in}{0.600000in}%
\pgfsys@useobject{currentmarker}{}%
\end{pgfscope}%
\end{pgfscope}%
\begin{pgfscope}%
\pgfsetbuttcap%
\pgfsetroundjoin%
\definecolor{currentfill}{rgb}{0.000000,0.000000,0.000000}%
\pgfsetfillcolor{currentfill}%
\pgfsetlinewidth{0.501875pt}%
\definecolor{currentstroke}{rgb}{0.000000,0.000000,0.000000}%
\pgfsetstrokecolor{currentstroke}%
\pgfsetdash{}{0pt}%
\pgfsys@defobject{currentmarker}{\pgfqpoint{0.000000in}{-0.055556in}}{\pgfqpoint{0.000000in}{0.000000in}}{%
\pgfpathmoveto{\pgfqpoint{0.000000in}{0.000000in}}%
\pgfpathlineto{\pgfqpoint{0.000000in}{-0.055556in}}%
\pgfusepath{stroke,fill}%
}%
\begin{pgfscope}%
\pgfsys@transformshift{6.166667in}{5.400000in}%
\pgfsys@useobject{currentmarker}{}%
\end{pgfscope}%
\end{pgfscope}%
\begin{pgfscope}%
\pgftext[x=6.166667in,y=0.544444in,,top]{\sffamily\fontsize{12.000000}{14.400000}\selectfont 5}%
\end{pgfscope}%
\begin{pgfscope}%
\pgfsetbuttcap%
\pgfsetroundjoin%
\definecolor{currentfill}{rgb}{0.000000,0.000000,0.000000}%
\pgfsetfillcolor{currentfill}%
\pgfsetlinewidth{0.501875pt}%
\definecolor{currentstroke}{rgb}{0.000000,0.000000,0.000000}%
\pgfsetstrokecolor{currentstroke}%
\pgfsetdash{}{0pt}%
\pgfsys@defobject{currentmarker}{\pgfqpoint{0.000000in}{0.000000in}}{\pgfqpoint{0.000000in}{0.055556in}}{%
\pgfpathmoveto{\pgfqpoint{0.000000in}{0.000000in}}%
\pgfpathlineto{\pgfqpoint{0.000000in}{0.055556in}}%
\pgfusepath{stroke,fill}%
}%
\begin{pgfscope}%
\pgfsys@transformshift{7.200000in}{0.600000in}%
\pgfsys@useobject{currentmarker}{}%
\end{pgfscope}%
\end{pgfscope}%
\begin{pgfscope}%
\pgfsetbuttcap%
\pgfsetroundjoin%
\definecolor{currentfill}{rgb}{0.000000,0.000000,0.000000}%
\pgfsetfillcolor{currentfill}%
\pgfsetlinewidth{0.501875pt}%
\definecolor{currentstroke}{rgb}{0.000000,0.000000,0.000000}%
\pgfsetstrokecolor{currentstroke}%
\pgfsetdash{}{0pt}%
\pgfsys@defobject{currentmarker}{\pgfqpoint{0.000000in}{-0.055556in}}{\pgfqpoint{0.000000in}{0.000000in}}{%
\pgfpathmoveto{\pgfqpoint{0.000000in}{0.000000in}}%
\pgfpathlineto{\pgfqpoint{0.000000in}{-0.055556in}}%
\pgfusepath{stroke,fill}%
}%
\begin{pgfscope}%
\pgfsys@transformshift{7.200000in}{5.400000in}%
\pgfsys@useobject{currentmarker}{}%
\end{pgfscope}%
\end{pgfscope}%
\begin{pgfscope}%
\pgftext[x=7.200000in,y=0.544444in,,top]{\sffamily\fontsize{12.000000}{14.400000}\selectfont 6}%
\end{pgfscope}%
\begin{pgfscope}%
\pgftext[x=4.100000in,y=0.313705in,,top]{\sffamily\fontsize{12.000000}{14.400000}\selectfont Refinement level}%
\end{pgfscope}%
\begin{pgfscope}%
\pgfsetbuttcap%
\pgfsetroundjoin%
\definecolor{currentfill}{rgb}{0.000000,0.000000,0.000000}%
\pgfsetfillcolor{currentfill}%
\pgfsetlinewidth{0.501875pt}%
\definecolor{currentstroke}{rgb}{0.000000,0.000000,0.000000}%
\pgfsetstrokecolor{currentstroke}%
\pgfsetdash{}{0pt}%
\pgfsys@defobject{currentmarker}{\pgfqpoint{0.000000in}{0.000000in}}{\pgfqpoint{0.055556in}{0.000000in}}{%
\pgfpathmoveto{\pgfqpoint{0.000000in}{0.000000in}}%
\pgfpathlineto{\pgfqpoint{0.055556in}{0.000000in}}%
\pgfusepath{stroke,fill}%
}%
\begin{pgfscope}%
\pgfsys@transformshift{1.000000in}{0.600000in}%
\pgfsys@useobject{currentmarker}{}%
\end{pgfscope}%
\end{pgfscope}%
\begin{pgfscope}%
\pgfsetbuttcap%
\pgfsetroundjoin%
\definecolor{currentfill}{rgb}{0.000000,0.000000,0.000000}%
\pgfsetfillcolor{currentfill}%
\pgfsetlinewidth{0.501875pt}%
\definecolor{currentstroke}{rgb}{0.000000,0.000000,0.000000}%
\pgfsetstrokecolor{currentstroke}%
\pgfsetdash{}{0pt}%
\pgfsys@defobject{currentmarker}{\pgfqpoint{-0.055556in}{0.000000in}}{\pgfqpoint{0.000000in}{0.000000in}}{%
\pgfpathmoveto{\pgfqpoint{0.000000in}{0.000000in}}%
\pgfpathlineto{\pgfqpoint{-0.055556in}{0.000000in}}%
\pgfusepath{stroke,fill}%
}%
\begin{pgfscope}%
\pgfsys@transformshift{7.200000in}{0.600000in}%
\pgfsys@useobject{currentmarker}{}%
\end{pgfscope}%
\end{pgfscope}%
\begin{pgfscope}%
\pgftext[x=0.944444in,y=0.600000in,right,]{\sffamily\fontsize{12.000000}{14.400000}\selectfont 0.000000}%
\end{pgfscope}%
\begin{pgfscope}%
\pgfsetbuttcap%
\pgfsetroundjoin%
\definecolor{currentfill}{rgb}{0.000000,0.000000,0.000000}%
\pgfsetfillcolor{currentfill}%
\pgfsetlinewidth{0.501875pt}%
\definecolor{currentstroke}{rgb}{0.000000,0.000000,0.000000}%
\pgfsetstrokecolor{currentstroke}%
\pgfsetdash{}{0pt}%
\pgfsys@defobject{currentmarker}{\pgfqpoint{0.000000in}{0.000000in}}{\pgfqpoint{0.055556in}{0.000000in}}{%
\pgfpathmoveto{\pgfqpoint{0.000000in}{0.000000in}}%
\pgfpathlineto{\pgfqpoint{0.055556in}{0.000000in}}%
\pgfusepath{stroke,fill}%
}%
\begin{pgfscope}%
\pgfsys@transformshift{1.000000in}{1.200000in}%
\pgfsys@useobject{currentmarker}{}%
\end{pgfscope}%
\end{pgfscope}%
\begin{pgfscope}%
\pgfsetbuttcap%
\pgfsetroundjoin%
\definecolor{currentfill}{rgb}{0.000000,0.000000,0.000000}%
\pgfsetfillcolor{currentfill}%
\pgfsetlinewidth{0.501875pt}%
\definecolor{currentstroke}{rgb}{0.000000,0.000000,0.000000}%
\pgfsetstrokecolor{currentstroke}%
\pgfsetdash{}{0pt}%
\pgfsys@defobject{currentmarker}{\pgfqpoint{-0.055556in}{0.000000in}}{\pgfqpoint{0.000000in}{0.000000in}}{%
\pgfpathmoveto{\pgfqpoint{0.000000in}{0.000000in}}%
\pgfpathlineto{\pgfqpoint{-0.055556in}{0.000000in}}%
\pgfusepath{stroke,fill}%
}%
\begin{pgfscope}%
\pgfsys@transformshift{7.200000in}{1.200000in}%
\pgfsys@useobject{currentmarker}{}%
\end{pgfscope}%
\end{pgfscope}%
\begin{pgfscope}%
\pgftext[x=0.944444in,y=1.200000in,right,]{\sffamily\fontsize{12.000000}{14.400000}\selectfont 0.000002}%
\end{pgfscope}%
\begin{pgfscope}%
\pgfsetbuttcap%
\pgfsetroundjoin%
\definecolor{currentfill}{rgb}{0.000000,0.000000,0.000000}%
\pgfsetfillcolor{currentfill}%
\pgfsetlinewidth{0.501875pt}%
\definecolor{currentstroke}{rgb}{0.000000,0.000000,0.000000}%
\pgfsetstrokecolor{currentstroke}%
\pgfsetdash{}{0pt}%
\pgfsys@defobject{currentmarker}{\pgfqpoint{0.000000in}{0.000000in}}{\pgfqpoint{0.055556in}{0.000000in}}{%
\pgfpathmoveto{\pgfqpoint{0.000000in}{0.000000in}}%
\pgfpathlineto{\pgfqpoint{0.055556in}{0.000000in}}%
\pgfusepath{stroke,fill}%
}%
\begin{pgfscope}%
\pgfsys@transformshift{1.000000in}{1.800000in}%
\pgfsys@useobject{currentmarker}{}%
\end{pgfscope}%
\end{pgfscope}%
\begin{pgfscope}%
\pgfsetbuttcap%
\pgfsetroundjoin%
\definecolor{currentfill}{rgb}{0.000000,0.000000,0.000000}%
\pgfsetfillcolor{currentfill}%
\pgfsetlinewidth{0.501875pt}%
\definecolor{currentstroke}{rgb}{0.000000,0.000000,0.000000}%
\pgfsetstrokecolor{currentstroke}%
\pgfsetdash{}{0pt}%
\pgfsys@defobject{currentmarker}{\pgfqpoint{-0.055556in}{0.000000in}}{\pgfqpoint{0.000000in}{0.000000in}}{%
\pgfpathmoveto{\pgfqpoint{0.000000in}{0.000000in}}%
\pgfpathlineto{\pgfqpoint{-0.055556in}{0.000000in}}%
\pgfusepath{stroke,fill}%
}%
\begin{pgfscope}%
\pgfsys@transformshift{7.200000in}{1.800000in}%
\pgfsys@useobject{currentmarker}{}%
\end{pgfscope}%
\end{pgfscope}%
\begin{pgfscope}%
\pgftext[x=0.944444in,y=1.800000in,right,]{\sffamily\fontsize{12.000000}{14.400000}\selectfont 0.000004}%
\end{pgfscope}%
\begin{pgfscope}%
\pgfsetbuttcap%
\pgfsetroundjoin%
\definecolor{currentfill}{rgb}{0.000000,0.000000,0.000000}%
\pgfsetfillcolor{currentfill}%
\pgfsetlinewidth{0.501875pt}%
\definecolor{currentstroke}{rgb}{0.000000,0.000000,0.000000}%
\pgfsetstrokecolor{currentstroke}%
\pgfsetdash{}{0pt}%
\pgfsys@defobject{currentmarker}{\pgfqpoint{0.000000in}{0.000000in}}{\pgfqpoint{0.055556in}{0.000000in}}{%
\pgfpathmoveto{\pgfqpoint{0.000000in}{0.000000in}}%
\pgfpathlineto{\pgfqpoint{0.055556in}{0.000000in}}%
\pgfusepath{stroke,fill}%
}%
\begin{pgfscope}%
\pgfsys@transformshift{1.000000in}{2.400000in}%
\pgfsys@useobject{currentmarker}{}%
\end{pgfscope}%
\end{pgfscope}%
\begin{pgfscope}%
\pgfsetbuttcap%
\pgfsetroundjoin%
\definecolor{currentfill}{rgb}{0.000000,0.000000,0.000000}%
\pgfsetfillcolor{currentfill}%
\pgfsetlinewidth{0.501875pt}%
\definecolor{currentstroke}{rgb}{0.000000,0.000000,0.000000}%
\pgfsetstrokecolor{currentstroke}%
\pgfsetdash{}{0pt}%
\pgfsys@defobject{currentmarker}{\pgfqpoint{-0.055556in}{0.000000in}}{\pgfqpoint{0.000000in}{0.000000in}}{%
\pgfpathmoveto{\pgfqpoint{0.000000in}{0.000000in}}%
\pgfpathlineto{\pgfqpoint{-0.055556in}{0.000000in}}%
\pgfusepath{stroke,fill}%
}%
\begin{pgfscope}%
\pgfsys@transformshift{7.200000in}{2.400000in}%
\pgfsys@useobject{currentmarker}{}%
\end{pgfscope}%
\end{pgfscope}%
\begin{pgfscope}%
\pgftext[x=0.944444in,y=2.400000in,right,]{\sffamily\fontsize{12.000000}{14.400000}\selectfont 0.000006}%
\end{pgfscope}%
\begin{pgfscope}%
\pgfsetbuttcap%
\pgfsetroundjoin%
\definecolor{currentfill}{rgb}{0.000000,0.000000,0.000000}%
\pgfsetfillcolor{currentfill}%
\pgfsetlinewidth{0.501875pt}%
\definecolor{currentstroke}{rgb}{0.000000,0.000000,0.000000}%
\pgfsetstrokecolor{currentstroke}%
\pgfsetdash{}{0pt}%
\pgfsys@defobject{currentmarker}{\pgfqpoint{0.000000in}{0.000000in}}{\pgfqpoint{0.055556in}{0.000000in}}{%
\pgfpathmoveto{\pgfqpoint{0.000000in}{0.000000in}}%
\pgfpathlineto{\pgfqpoint{0.055556in}{0.000000in}}%
\pgfusepath{stroke,fill}%
}%
\begin{pgfscope}%
\pgfsys@transformshift{1.000000in}{3.000000in}%
\pgfsys@useobject{currentmarker}{}%
\end{pgfscope}%
\end{pgfscope}%
\begin{pgfscope}%
\pgfsetbuttcap%
\pgfsetroundjoin%
\definecolor{currentfill}{rgb}{0.000000,0.000000,0.000000}%
\pgfsetfillcolor{currentfill}%
\pgfsetlinewidth{0.501875pt}%
\definecolor{currentstroke}{rgb}{0.000000,0.000000,0.000000}%
\pgfsetstrokecolor{currentstroke}%
\pgfsetdash{}{0pt}%
\pgfsys@defobject{currentmarker}{\pgfqpoint{-0.055556in}{0.000000in}}{\pgfqpoint{0.000000in}{0.000000in}}{%
\pgfpathmoveto{\pgfqpoint{0.000000in}{0.000000in}}%
\pgfpathlineto{\pgfqpoint{-0.055556in}{0.000000in}}%
\pgfusepath{stroke,fill}%
}%
\begin{pgfscope}%
\pgfsys@transformshift{7.200000in}{3.000000in}%
\pgfsys@useobject{currentmarker}{}%
\end{pgfscope}%
\end{pgfscope}%
\begin{pgfscope}%
\pgftext[x=0.944444in,y=3.000000in,right,]{\sffamily\fontsize{12.000000}{14.400000}\selectfont 0.000008}%
\end{pgfscope}%
\begin{pgfscope}%
\pgfsetbuttcap%
\pgfsetroundjoin%
\definecolor{currentfill}{rgb}{0.000000,0.000000,0.000000}%
\pgfsetfillcolor{currentfill}%
\pgfsetlinewidth{0.501875pt}%
\definecolor{currentstroke}{rgb}{0.000000,0.000000,0.000000}%
\pgfsetstrokecolor{currentstroke}%
\pgfsetdash{}{0pt}%
\pgfsys@defobject{currentmarker}{\pgfqpoint{0.000000in}{0.000000in}}{\pgfqpoint{0.055556in}{0.000000in}}{%
\pgfpathmoveto{\pgfqpoint{0.000000in}{0.000000in}}%
\pgfpathlineto{\pgfqpoint{0.055556in}{0.000000in}}%
\pgfusepath{stroke,fill}%
}%
\begin{pgfscope}%
\pgfsys@transformshift{1.000000in}{3.600000in}%
\pgfsys@useobject{currentmarker}{}%
\end{pgfscope}%
\end{pgfscope}%
\begin{pgfscope}%
\pgfsetbuttcap%
\pgfsetroundjoin%
\definecolor{currentfill}{rgb}{0.000000,0.000000,0.000000}%
\pgfsetfillcolor{currentfill}%
\pgfsetlinewidth{0.501875pt}%
\definecolor{currentstroke}{rgb}{0.000000,0.000000,0.000000}%
\pgfsetstrokecolor{currentstroke}%
\pgfsetdash{}{0pt}%
\pgfsys@defobject{currentmarker}{\pgfqpoint{-0.055556in}{0.000000in}}{\pgfqpoint{0.000000in}{0.000000in}}{%
\pgfpathmoveto{\pgfqpoint{0.000000in}{0.000000in}}%
\pgfpathlineto{\pgfqpoint{-0.055556in}{0.000000in}}%
\pgfusepath{stroke,fill}%
}%
\begin{pgfscope}%
\pgfsys@transformshift{7.200000in}{3.600000in}%
\pgfsys@useobject{currentmarker}{}%
\end{pgfscope}%
\end{pgfscope}%
\begin{pgfscope}%
\pgftext[x=0.944444in,y=3.600000in,right,]{\sffamily\fontsize{12.000000}{14.400000}\selectfont 0.000010}%
\end{pgfscope}%
\begin{pgfscope}%
\pgfsetbuttcap%
\pgfsetroundjoin%
\definecolor{currentfill}{rgb}{0.000000,0.000000,0.000000}%
\pgfsetfillcolor{currentfill}%
\pgfsetlinewidth{0.501875pt}%
\definecolor{currentstroke}{rgb}{0.000000,0.000000,0.000000}%
\pgfsetstrokecolor{currentstroke}%
\pgfsetdash{}{0pt}%
\pgfsys@defobject{currentmarker}{\pgfqpoint{0.000000in}{0.000000in}}{\pgfqpoint{0.055556in}{0.000000in}}{%
\pgfpathmoveto{\pgfqpoint{0.000000in}{0.000000in}}%
\pgfpathlineto{\pgfqpoint{0.055556in}{0.000000in}}%
\pgfusepath{stroke,fill}%
}%
\begin{pgfscope}%
\pgfsys@transformshift{1.000000in}{4.200000in}%
\pgfsys@useobject{currentmarker}{}%
\end{pgfscope}%
\end{pgfscope}%
\begin{pgfscope}%
\pgfsetbuttcap%
\pgfsetroundjoin%
\definecolor{currentfill}{rgb}{0.000000,0.000000,0.000000}%
\pgfsetfillcolor{currentfill}%
\pgfsetlinewidth{0.501875pt}%
\definecolor{currentstroke}{rgb}{0.000000,0.000000,0.000000}%
\pgfsetstrokecolor{currentstroke}%
\pgfsetdash{}{0pt}%
\pgfsys@defobject{currentmarker}{\pgfqpoint{-0.055556in}{0.000000in}}{\pgfqpoint{0.000000in}{0.000000in}}{%
\pgfpathmoveto{\pgfqpoint{0.000000in}{0.000000in}}%
\pgfpathlineto{\pgfqpoint{-0.055556in}{0.000000in}}%
\pgfusepath{stroke,fill}%
}%
\begin{pgfscope}%
\pgfsys@transformshift{7.200000in}{4.200000in}%
\pgfsys@useobject{currentmarker}{}%
\end{pgfscope}%
\end{pgfscope}%
\begin{pgfscope}%
\pgftext[x=0.944444in,y=4.200000in,right,]{\sffamily\fontsize{12.000000}{14.400000}\selectfont 0.000012}%
\end{pgfscope}%
\begin{pgfscope}%
\pgfsetbuttcap%
\pgfsetroundjoin%
\definecolor{currentfill}{rgb}{0.000000,0.000000,0.000000}%
\pgfsetfillcolor{currentfill}%
\pgfsetlinewidth{0.501875pt}%
\definecolor{currentstroke}{rgb}{0.000000,0.000000,0.000000}%
\pgfsetstrokecolor{currentstroke}%
\pgfsetdash{}{0pt}%
\pgfsys@defobject{currentmarker}{\pgfqpoint{0.000000in}{0.000000in}}{\pgfqpoint{0.055556in}{0.000000in}}{%
\pgfpathmoveto{\pgfqpoint{0.000000in}{0.000000in}}%
\pgfpathlineto{\pgfqpoint{0.055556in}{0.000000in}}%
\pgfusepath{stroke,fill}%
}%
\begin{pgfscope}%
\pgfsys@transformshift{1.000000in}{4.800000in}%
\pgfsys@useobject{currentmarker}{}%
\end{pgfscope}%
\end{pgfscope}%
\begin{pgfscope}%
\pgfsetbuttcap%
\pgfsetroundjoin%
\definecolor{currentfill}{rgb}{0.000000,0.000000,0.000000}%
\pgfsetfillcolor{currentfill}%
\pgfsetlinewidth{0.501875pt}%
\definecolor{currentstroke}{rgb}{0.000000,0.000000,0.000000}%
\pgfsetstrokecolor{currentstroke}%
\pgfsetdash{}{0pt}%
\pgfsys@defobject{currentmarker}{\pgfqpoint{-0.055556in}{0.000000in}}{\pgfqpoint{0.000000in}{0.000000in}}{%
\pgfpathmoveto{\pgfqpoint{0.000000in}{0.000000in}}%
\pgfpathlineto{\pgfqpoint{-0.055556in}{0.000000in}}%
\pgfusepath{stroke,fill}%
}%
\begin{pgfscope}%
\pgfsys@transformshift{7.200000in}{4.800000in}%
\pgfsys@useobject{currentmarker}{}%
\end{pgfscope}%
\end{pgfscope}%
\begin{pgfscope}%
\pgftext[x=0.944444in,y=4.800000in,right,]{\sffamily\fontsize{12.000000}{14.400000}\selectfont 0.000014}%
\end{pgfscope}%
\begin{pgfscope}%
\pgfsetbuttcap%
\pgfsetroundjoin%
\definecolor{currentfill}{rgb}{0.000000,0.000000,0.000000}%
\pgfsetfillcolor{currentfill}%
\pgfsetlinewidth{0.501875pt}%
\definecolor{currentstroke}{rgb}{0.000000,0.000000,0.000000}%
\pgfsetstrokecolor{currentstroke}%
\pgfsetdash{}{0pt}%
\pgfsys@defobject{currentmarker}{\pgfqpoint{0.000000in}{0.000000in}}{\pgfqpoint{0.055556in}{0.000000in}}{%
\pgfpathmoveto{\pgfqpoint{0.000000in}{0.000000in}}%
\pgfpathlineto{\pgfqpoint{0.055556in}{0.000000in}}%
\pgfusepath{stroke,fill}%
}%
\begin{pgfscope}%
\pgfsys@transformshift{1.000000in}{5.400000in}%
\pgfsys@useobject{currentmarker}{}%
\end{pgfscope}%
\end{pgfscope}%
\begin{pgfscope}%
\pgfsetbuttcap%
\pgfsetroundjoin%
\definecolor{currentfill}{rgb}{0.000000,0.000000,0.000000}%
\pgfsetfillcolor{currentfill}%
\pgfsetlinewidth{0.501875pt}%
\definecolor{currentstroke}{rgb}{0.000000,0.000000,0.000000}%
\pgfsetstrokecolor{currentstroke}%
\pgfsetdash{}{0pt}%
\pgfsys@defobject{currentmarker}{\pgfqpoint{-0.055556in}{0.000000in}}{\pgfqpoint{0.000000in}{0.000000in}}{%
\pgfpathmoveto{\pgfqpoint{0.000000in}{0.000000in}}%
\pgfpathlineto{\pgfqpoint{-0.055556in}{0.000000in}}%
\pgfusepath{stroke,fill}%
}%
\begin{pgfscope}%
\pgfsys@transformshift{7.200000in}{5.400000in}%
\pgfsys@useobject{currentmarker}{}%
\end{pgfscope}%
\end{pgfscope}%
\begin{pgfscope}%
\pgftext[x=0.944444in,y=5.400000in,right,]{\sffamily\fontsize{12.000000}{14.400000}\selectfont 0.000016}%
\end{pgfscope}%
\begin{pgfscope}%
\pgftext[x=0.079753in,y=3.000000in,,bottom,rotate=90.000000]{\sffamily\fontsize{12.000000}{14.400000}\selectfont Functional}%
\end{pgfscope}%
\begin{pgfscope}%
\pgfsetbuttcap%
\pgfsetmiterjoin%
\definecolor{currentfill}{rgb}{1.000000,1.000000,1.000000}%
\pgfsetfillcolor{currentfill}%
\pgfsetlinewidth{1.003750pt}%
\definecolor{currentstroke}{rgb}{0.000000,0.000000,0.000000}%
\pgfsetstrokecolor{currentstroke}%
\pgfsetdash{}{0pt}%
\pgfpathmoveto{\pgfqpoint{5.662617in}{0.700000in}}%
\pgfpathlineto{\pgfqpoint{7.100000in}{0.700000in}}%
\pgfpathlineto{\pgfqpoint{7.100000in}{1.640664in}}%
\pgfpathlineto{\pgfqpoint{5.662617in}{1.640664in}}%
\pgfpathclose%
\pgfusepath{stroke,fill}%
\end{pgfscope}%
\begin{pgfscope}%
\pgfsetbuttcap%
\pgfsetroundjoin%
\pgfsetlinewidth{1.003750pt}%
\definecolor{currentstroke}{rgb}{0.000000,0.000000,0.000000}%
\pgfsetstrokecolor{currentstroke}%
\pgfsetdash{{6.000000pt}{6.000000pt}}{0.000000pt}%
\pgfpathmoveto{\pgfqpoint{5.802617in}{1.478711in}}%
\pgfpathlineto{\pgfqpoint{6.082617in}{1.478711in}}%
\pgfusepath{stroke}%
\end{pgfscope}%
\begin{pgfscope}%
\pgftext[x=6.302617in,y=1.408711in,left,base]{\sffamily\fontsize{14.400000}{17.280000}\selectfont Mesh 1}%
\end{pgfscope}%
\begin{pgfscope}%
\pgfsetbuttcap%
\pgfsetroundjoin%
\pgfsetlinewidth{1.003750pt}%
\definecolor{currentstroke}{rgb}{0.000000,0.000000,0.000000}%
\pgfsetstrokecolor{currentstroke}%
\pgfsetdash{{3.000000pt}{5.000000pt}{1.000000pt}{5.000000pt}}{0.000000pt}%
\pgfpathmoveto{\pgfqpoint{5.802617in}{1.185156in}}%
\pgfpathlineto{\pgfqpoint{6.082617in}{1.185156in}}%
\pgfusepath{stroke}%
\end{pgfscope}%
\begin{pgfscope}%
\pgftext[x=6.302617in,y=1.115156in,left,base]{\sffamily\fontsize{14.400000}{17.280000}\selectfont Mesh 2}%
\end{pgfscope}%
\begin{pgfscope}%
\pgfsetrectcap%
\pgfsetroundjoin%
\pgfsetlinewidth{1.003750pt}%
\definecolor{currentstroke}{rgb}{0.000000,0.000000,0.000000}%
\pgfsetstrokecolor{currentstroke}%
\pgfsetdash{}{0pt}%
\pgfpathmoveto{\pgfqpoint{5.802617in}{0.891601in}}%
\pgfpathlineto{\pgfqpoint{6.082617in}{0.891601in}}%
\pgfusepath{stroke}%
\end{pgfscope}%
\begin{pgfscope}%
\pgftext[x=6.302617in,y=0.821601in,left,base]{\sffamily\fontsize{14.400000}{17.280000}\selectfont Mesh 3}%
\end{pgfscope}%
\end{pgfpicture}%
\makeatother%
\endgroup%

%% file: topopt_els_arxiv_v2.bbl
\begin{thebibliography}{10}
\expandafter\ifx\csname url\endcsname\relax
  \def\url#1{\texttt{#1}}\fi
\expandafter\ifx\csname urlprefix\endcsname\relax\def\urlprefix{URL }\fi
\expandafter\ifx\csname href\endcsname\relax
  \def\href#1#2{#2} \def\path#1{#1}\fi

\bibitem{steven2000evolutionary}
G.~Steven, Q.~Li, Y.~Xie, Evolutionary topology and shape design for general
  physical field problems, Computational mechanics 26~(2) (2000) 129--139.

\bibitem{safonov}
A.~Safonov, J.~Jones, Physarum computing and topology optimisation,
  International Journal of Parallel, Emergent and Distributed Systems 0~(0)
  (2016) 1--18.
\newblock \href {http://dx.doi.org/10.1080/17445760.2016.1221073}
  {\path{doi:10.1080/17445760.2016.1221073}}.

\bibitem{gratsch2005}
T.~Gr{\"a}tsch, K.-J. Bathe, A posteriori error estimation techniques in
  practical finite element analysis, Computers \& structures 83~(4) (2005)
  235--265.

\bibitem{xie1993simple}
Y.~M. Xie, G.~P. Steven, A simple evolutionary procedure for structural
  optimization, Computers \& structures 49~(5) (1993) 885--896.

\bibitem{xie1997}
Y.~M. Xie, G.~P. Steven, Basic evolutionary structural optimization, Springer,
  1997.

\bibitem{querin1998}
O.~Querin, G.~Steven, Y.~Xie, Evolutionary structural optimisation (eso) using
  a bidirectional algorithm, Engineering Computations 15~(8) (1998) 1031--1048.

\bibitem{gao2008}
T.~Gao, W.~Zhang, J.~Zhu, Y.~Xu, D.~Bassir, Topology optimization of heat
  conduction problem involving design-dependent heat load effect, Finite
  Elements in Analysis and Design 44~(14) (2008) 805--813.

\bibitem{gersborg2006}
A.~Gersborg-Hansen, M.~P. Bends{\o}e, O.~Sigmund, Topology optimization of heat
  conduction problems using the finite volume method, Structural and
  multidisciplinary optimization 31~(4) (2006) 251--259.

\bibitem{zuo2005}
K.~Zuo, L.~Chen, Y.~Zhang, S.~Wang, Structural optimal design of heat
  conductive body with topology optimization method, Chin J Mech Eng 41~(4)
  (2005) 13--21.

\bibitem{FD2015}
F.~Rathgeber, D.~A. Ham, L.~Mitchell, M.~Lange, F.~Luporini, A.~T. McRae, G.-T.
  Bercea, G.~R. Markall, P.~H. Kelly, Firedrake: automating the finite element
  method by composing abstractions, arXiv preprint arXiv:1501.01809.

\bibitem{luporini2016}
F.~Luporini, D.~A. Ham, P.~H. Kelly, An algorithm for the optimization of
  finite element integration loops, arXiv preprint arXiv:1604.05872.

\bibitem{luporini2015}
F.~Luporini, A.~L. Varbanescu, F.~Rathgeber, G.-T. Bercea, J.~Ramanujam, D.~A.
  Ham, P.~H. Kelly, Cross-loop optimization of arithmetic intensity for finite
  element local assembly, ACM Transactions on Architecture and Code
  Optimization (TACO) 11~(4) (2015) 57.

\bibitem{mcrae2014}
A.~T. McRae, G.-T. Bercea, L.~Mitchell, D.~A. Ham, C.~J. Cotter, Automated
  generation and symbolic manipulation of tensor product finite elements, arXiv
  preprint arXiv:1411.2940.

\bibitem{gmsh}
C.~Geuzaine, J.-F. Remacle, Gmsh: A 3-d finite element mesh generator with
  built-in pre-and post-processing facilities, International Journal for
  Numerical Methods in Engineering 79~(11) (2009) 1309--1331.

\bibitem{lau2001convex}
G.~Lau, H.~Du, M.~Lim, Convex analysis for topology optimization of compliant
  mechanisms, Structural and Multidisciplinary Optimization 22~(4) (2001)
  284--294.

\bibitem{sigmund2001}
O.~Sigmund, A 99 line topology optimization code written in matlab, Structural
  and multidisciplinary optimization 21~(2) (2001) 120--127.

\bibitem{deaton2014survey}
J.~D. Deaton, R.~V. Grandhi, A survey of structural and multidisciplinary
  continuum topology optimization: post 2000, Structural and Multidisciplinary
  Optimization 49~(1) (2014) 1--38.

\bibitem{bourdin2001filters}
B.~Bourdin, Filters in topology optimization, International Journal for
  Numerical Methods in Engineering 50~(9) (2001) 2143--2158.

\bibitem{bruns2001topology}
T.~E. Bruns, D.~A. Tortorelli, Topology optimization of non-linear elastic
  structures and compliant mechanisms, Computer Methods in Applied Mechanics
  and Engineering 190~(26) (2001) 3443--3459.

\bibitem{sigmund1997}
O.~Sigmund, On the design of compliant mechanisms using topology optimization*,
  Journal of Structural Mechanics 25~(4) (1997) 493--524.

\bibitem{lazarov2011filters}
B.~S. Lazarov, O.~Sigmund, Filters in topology optimization based on
  helmholtz-type differential equations, International Journal for Numerical
  Methods in Engineering 86~(6) (2011) 765--781.

\bibitem{kawamoto2011heaviside}
A.~Kawamoto, T.~Matsumori, S.~Yamasaki, T.~Nomura, T.~Kondoh, S.~Nishiwaki,
  Heaviside projection based topology optimization by a pde-filtered scalar
  function, Structural and Multidisciplinary Optimization 44~(1) (2011) 19--24.

\bibitem{guest2004achieving}
J.~K. Guest, J.~H. Pr{\'e}vost, T.~Belytschko, Achieving minimum length scale
  in topology optimization using nodal design variables and projection
  functions, International journal for numerical methods in engineering 61~(2)
  (2004) 238--254.

\bibitem{sigmund2007morphology}
O.~Sigmund, Morphology-based black and white filters for topology optimization,
  Structural and Multidisciplinary Optimization 33~(4-5) (2007) 401--424.

\bibitem{xu2010volume}
S.~Xu, Y.~Cai, G.~Cheng, Volume preserving nonlinear density filter based on
  heaviside functions, Structural and Multidisciplinary Optimization 41~(4)
  (2010) 495--505.

\bibitem{guest2011eliminating}
J.~K. Guest, A.~Asadpoure, S.-H. Ha, Eliminating beta-continuation from
  heaviside projection and density filter algorithms, Structural and
  Multidisciplinary Optimization 44~(4) (2011) 443--453.

\bibitem{duhring2010design}
M.~B. D{\"u}hring, O.~Sigmund, T.~Feurer, Design of photonic bandgap fibers by
  topology optimization, JOSA B 27~(1) (2010) 51--58.

\bibitem{ambrosio1993optimal}
L.~Ambrosio, G.~Buttazzo, An optimal design problem with perimeter
  penalization, Calculus of Variations and Partial Differential Equations 1~(1)
  (1993) 55--69.

\bibitem{haber1996perimeter}
R.~Haber, M.~Bend{\o}se, C.~Jog, Perimeter constrained topology optimization of
  continuum structures, in: IUTAM Symposium on Optimization of Mechanical
  Systems, Springer, 1996, pp. 113--120.

\bibitem{niordson1983optimal}
F.~Niordson, Optimal design of elastic plates with a constraint on the slope of
  the thickness function, International Journal of Solids and Structures 19~(2)
  (1983) 141--151.

\bibitem{petersson1998slope}
J.~Petersson, O.~Sigmund, Slope constrained topology optimization,
  International Journal for Numerical Methods in Engineering 41~(8) (1998)
  1417--1434.

\bibitem{borrvall2001topology}
T.~Borrvall, J.~Petersson, Topology optimization using regularized intermediate
  density control, Computer Methods in Applied Mechanics and Engineering
  190~(37) (2001) 4911--4928.

\bibitem{allaire2004structural}
G.~Allaire, F.~Jouve, A.-M. Toader, Structural optimization using sensitivity
  analysis and a level-set method, Journal of computational physics 194~(1)
  (2004) 363--393.

\bibitem{sigmund2013topology}
O.~Sigmund, K.~Maute, Topology optimization approaches, Structural and
  Multidisciplinary Optimization 48~(6) (2013) 1031--1055.

\end{thebibliography}
